\documentclass[11pt]{article}
\usepackage{amsfonts}
\usepackage{amsmath}
\usepackage{amsthm}
\usepackage{latexsym}
\usepackage{verbatim}
\usepackage[utf8]{inputenc}
\usepackage[T1]{fontenc}
\usepackage{lmodern}
\usepackage{a4wide}
\usepackage{hyperref}

\newcommand{\intO} {\int_\Omega} 
\newcommand{\ve} {\mathbf{v}} 
\newcommand{\dd} {\,\mathrm{d}}
\newcommand{\dyw}{\mathrm{div}\,}
\newcommand{\rot}{\mathrm{curl}\,} 

\newcommand{\essinf}{\mathrm{ess\, inf}\,}
\newcommand{\esssup}{\mathrm{ess\, sup}\,}

\newcommand{\fr}{\mathbf{F}} 
\newcommand{\vfr}{\mathbf{G}} 
\newcommand{\sty}{\mathbf{\boldsymbol \tau}}
\newcommand{\norm}{\mathbf{n}}
\newcommand{\we}{\mathbf{w}}
\newcommand{\brzeg}{{\partial \Omega}} 
\newcommand{\slabo}{\rightharpoonup}
\newtheorem{tw}{Theorem}[section] 
\newtheorem{defn}[tw]{Definition}
\newtheorem{stw}[tw]{Proposition}
\newtheorem{uwaga}[tw]{Remark} 
\newtheorem{lemat}[tw]{Lemma}
\newtheorem{wniosek}[tw]{Corollary}
 
\numberwithin{equation}{section}

\title{On steady solutions to vacuumless \\Newtonian models of compressible flow} 
\author{Micha{\l} {\L}asica \\
\small Institute of Applied Mathematics and Mechanics, University of Warsaw \\ 
\small Banacha 2, 02-097 Warszawa, Poland \\
\small \tt lasica@mimuw.edu.pl } 
\date{}
\begin{document}
\maketitle
\begin{abstract} 
We prove the existence of weak solutions to the steady compressible Navier-Stokes system in the barotropic case for a class of pressure laws singular at vacuum. We consider the problem in a bounded domain in $\mathbb R^2$ with slip boundary conditions. Due to appropriate construction of approximate solutions used in proof, obtained density is bounded away from $0$ (and infinity). Owing to a classical result by P.-L. Lions, this implies that density and gradient of velocity are at least H\" older continuous, which does not generally hold for the classical isentropic model in the presence of vacuum. 
\end{abstract} 
\smallskip
\begin{center}
\end{center}

\section{Introduction} 
We consider the Navier-Stokes system of equations describing stationary flow of a compressible fluid. In this paper we restrict ourselves to the case of flow in a bounded domain $\Omega$ in $\mathbb R^2$. The state of fluid is described by vector field $\ve \colon \Omega \to \mathbb{R}^2$ representing fluid velocity and scalar fields $\varrho \colon \Omega \to \mathbb{R}_+$, $\pi \colon \Omega \to \mathbb{R}$ corresponding to density and pressure, respectively. These quantities satisfy local laws of conservation of momentum and mass
\begin{equation}
\dyw (\varrho \ve \otimes \ve) = \dyw \mathbb T +\varrho \fr + \vfr, 
\label{ped} 
\end{equation}
\begin{equation} 
\dyw(\varrho \ve) = 0 \qquad \text{in } \Omega,  
\label{masa} 
\end{equation} 
where $\fr, \vfr \colon \Omega \to \mathbb{R}^2$ represent body forces per unit mass and unit volume respectively and $\mathbb T$ is the stress tensor. In the Newtonian case, which is our interest here, $\mathbb T = \mathbb{T}(\ve, \pi) = 2 \mu \mathbb{D}(\ve) + (\nu \dyw \ve - \pi) \mathbb{I} = \mu (\nabla \ve + \nabla \ve^T) + (\nu \dyw \ve - \pi) \mathbb{I}$, where $\mu$ and $\nu$ are the Lam\'e constants describing viscosity of the fluid. 

In order to obtain a well-posed problem, additional requirements on solutions (corresponding to characterisation of thermodynamical properties of the fluid) should be provided. We restrict ourselves to the barotropic regime, where pressure is a given function of density. Physically, this corresponds to disregarding heat conduction in the fluid. In particular, we study the case of general pressure law singular at vacuum, i.\,e. 
\begin{equation} 
\lim_{\varrho \to 0^+} \pi(\varrho) =- \infty .
\label{spadek} 
\end{equation} 
The paper may be seen as an answer to the question of required behaviour of $\pi$ near vacuum that would guarantee boundedness of the density away from $0$. As we will see, any (otherwise admissible) pressure law satisfying \eqref{spadek} will suffice. We note here that pressure laws of this form (given by negative-power law near vacuum) were used by Bresch and Desjardins \cite{bd}, who proved stability for non-stationary Navier-Stokes-Fourier system with pointwise energy balance (with viscosity coefficients vanishing at vacuum), and also by Zatorska \cite{ika}. Such form of pressure law may be seen (see \cite{bd}) as a way to retain properties of a medium whose temperature tends to $0$ (which also in the isentropic regime needs to be the case when density approaches $0$). Let us now discuss briefly some aspects of the state of theory of compressible flow in order to justify dealing with such model. 

\subsection{Discussion of the theory}

The stationary compressible Navier-Stokes system, describing steady flow of Newtonian gas, should not be considered a special case of the evolutionary system. In fact, it presents some specific mathematical difficulties and was investigated by many authors for its own sake. 

First mathematical results for the steady system were in the barotropic regime by Lions \cite{lions}. He proved the existence in bounded domains for any smooth (on the closed half-line $[0, \infty[$) monotone pressure law $\pi = \pi(\varrho)$ behaving near infinity like $\varrho^\gamma$ with $\gamma > {5 \over 3}$ in $\mathbb R^3$ or with $\gamma > 1$ in $\mathbb R^2$. This result was later repeatedly strengthened by considering lower, more physically relevant values of $\gamma$. One of latest works in this direction was by Frehse, Steinhauer and Weigant \cite{fsw}, who obtained estimates that allowed to treat the case $\gamma > \frac{4}{3}$ in 3D. Then, most recently, Jiang and Zhou proved the existence for $\gamma > 1$ \cite{jz} in 3D with periodic boundary conditions, which was later repeated by Novotn\' y and Jessle for Dirichlet boundary conditions. The heat conducting case was considered by Novotn\' y and Pokorn\' y, separately in 3D \cite{np1, np3} and 2D \cite{np2, np4}. However, the case of 
pressure singular at vacuum was not treated until now. As for regularity, we note that $L^\infty$ estimates for $\varrho$ are known in some cases, including the barotropic case in 2D \cite{lions}, see also \cite{mp}. However, in the presence of vacuum, the density (and also the gradient of velocity) may be discontinuous even for smooth data (at least assuming $\vfr \neq 0$), as shown by an example due to Lions. On the other hand, it was also proved by Lions that if $\varrho$ is bounded away from $0$, $\varrho$ and $\nabla \ve$ need to be at least H\" older continuous. 

One of motivations to study the steady Navier-Stokes system is that it has very similar structure and properties to the system obtained by time-discretisation of the non-stationary (evolutionary) Navier-Stokes equations \cite{lions} (see also \cite{semi}). The existence theory of solutions to the equations of evolutionary flow is quite advanced. Here, we only mention the seminal work of Lions \cite{lions} and the paper \cite{fnp} by Feireisl, Novotn\' y and Petzeltova, where the existence of weak solutions is proved in the barotropic case with $\pi(\varrho) = \varrho^\gamma$, $\gamma > \frac{3}{2}$. However, not much is known about regularity of solutions to the evolutionary system. Global existence of strong/classical solutions is only known in 1D (see e.\,g.\,\cite{ks, wat, lw}). In 2D there is only a result by Kazhikhov and Weigant \cite{vk}, where the existence of classical solutions is proved assuming that $\nu = \nu(\varrho) = \varrho^3$. We note that in \cite{vk} (as in \cite{ks} etc.) it is assumed 
that the initial density is bounded away from $0$. Lack of this condition may cause, generally speaking, ill-posedness of the problem (see e.\,g.\;\cite{failure}). 

There are however some conditional regularity results. In particular, a recent result by Sun et al. \cite{sun2d, sun3d} implies that the solutions to the evolutionary compressible Navier-Stokes system are regular as long as the density is bounded in $L^\infty$. Such a bound is as for now known only in some very specific cases, e.\,g.\;\cite{fprs}. That is why time-discretisation, with guaranteed $L^\infty$ bound for a single stationary equation, may be useful in investigating the evolutionary system. As we noticed, for solutions to the stationary system to be regular, they should not admit vacuum. We stress here, that even if there are stationary solutions for given external forces which exhibit vacuum, it does not imply that vacuum states will be created in the evolution from vacuumless initial data. In fact, it is prohibited in some cases \cite{hs, dz, xy}. 

Finally, we mention a connection with a noted result by Seregin and {\v{S}}ver{\'a}k \cite{ss} for incompressible 3D Navier-Stokes system. The incompressible system does (in principle) allow negative infinite pressure, in contrast to standard compressible models. The result \cite{ss} implies that as long as pressure is bounded from below, the solutions to evolutionary problem are regular. Here, we allow the barotropic pressure to attain $- \infty$ and show that (in the steady 2D case) it will always be bounded from below. 

\subsection{Statement of results}
Let us now state our results precisely. We make standard assumptions that 
\begin{equation} 
\pi  \text{ is non-decreasing and belongs to } C^1(\mathbb R_+) 
\label{cisnienie}
\end{equation}
and that growth condition 
\begin{equation} 
\pi(\varrho)\geq a \varrho^\gamma  \quad \text{for all } \varrho \geq \varrho_1
\label{wzrost}
\end{equation}
is satisfied with given $\gamma >1$, $a > 0$, $\varrho_1 > 0$. As the pressure is defined up to a constant, we may assume that $\pi(\varrho_0) =0$ for some $\varrho_0 > 0$. In particular, due to \eqref{spadek} and \eqref{wzrost} we may (and will) choose $\varrho_0\leq {h \over 4}$ so that $\pi(\varrho) \geq a(\varrho^\gamma - \varrho_0^\gamma) \quad \text{for all } \varrho \geq \varrho_0$. Further, we define $\pi_\pm = \max(\pm \pi, 0)$ so that $\pi = \pi_+ - \pi_-$ and 
\begin{equation}
\pi_+(\varrho) \geq a(\varrho^\gamma - \varrho_0^\gamma) \quad \text{for all } \varrho \in \mathbb R_+. 
\label{wzrost2}
\end{equation}

We only treat the case of slip boundary conditions, i.\,e. 
\begin{equation}
\ve \cdot \norm = 0 \qquad \text{on } \partial \Omega, 
\label{sciany}
\end{equation} 
\begin{equation} 
\norm \cdot \mathbb{T}(\ve, \pi) \cdot \sty + f\ve \cdot \sty = 0 \qquad \text{on } \partial \Omega, 
\label{slizg}
\end{equation}
where $\norm$ and $\sty$ are vectors normal and tangent to the boundary respectively and $f$ is the coefficient of friction. The form of tensor $\mathbb T$ implies that $\norm \cdot \mathbb{T}(\ve, \pi) \cdot \sty = 2 \mu\, \norm \cdot \mathbb{D}(\ve) \cdot \sty$ (in particular, pressure $\pi$ does not actually appear in the formula \eqref{slizg}). 

We also assume that finite total mass $M$ (or mean density $h = {M \over |\Omega|}$) of the fluid is prescribed, i.\,e. 
\begin{equation} 
\intO \varrho \dd x = M = h |\Omega| > 0 .
\label{h}
\end{equation}

The following notion of weak solutions to (\ref{ped}-\ref{slizg}) is introduced. 
\begin{defn}
We call a pair $(\varrho, \ve) \in L^1(\Omega) \times W^{1,2} (\Omega)$ satisfying conditions $\pi_+(\rho), \pi_-(\rho) \in L^1(\Omega)$, $\ve \cdot \norm = 0$ on $\partial \Omega$ in the trace sense a weak solution to (\ref{ped}-\ref{slizg}) if 
$$\intO \varrho \ve \cdot \nabla \eta \dd x = 0 \quad \text{ for any }\eta \in C^{\infty} (\overline{\Omega}),$$
\begin{multline*} 
- \intO \varrho \ve \otimes \ve \colon \nabla \mathbf{\boldsymbol \varphi} \dd x + 2 \mu \intO \mathbb{D}(\ve) \colon \mathbb{D}(\mathbf{\boldsymbol \varphi}) \dd x + \nu \intO \dyw \ve \dyw \mathbf{\boldsymbol \varphi} \dd x \\ + \int_{\partial \Omega} f (\ve \cdot \sty) (\mathbf{\boldsymbol \varphi} \cdot \sty) \dd S - \intO \pi (\varrho) \dyw \mathbf{\boldsymbol \varphi} \dd x  = \intO (\varrho \fr + \vfr) \cdot \mathbf{\boldsymbol \varphi} \dd x \\ \quad \text{ for any } \mathbf{\boldsymbol \varphi} \in C^{\infty} (\overline{\Omega})^2 \text{ s.\,t. } \mathbf{\boldsymbol \varphi} \cdot \norm =0 \text{ on } \partial \Omega.
\end{multline*} 
\label{thedef}
\end{defn}
For the above definition of weak solutions it is possible to show existence. 
\begin{tw} 
Let $\Omega \in \mathbb{R}^2$ be a bounded domain with boundary of class $C^2$. Let $\mu >0$, $2 \mu + 3\nu >0$, $f \geq 0$, $\fr, \vfr \in L^\infty(\Omega)$, $h>0$ and $\pi$ is as in (\ref{spadek}-\ref{wzrost}). Then, there exists a weak solution $(\varrho, \ve)$ to (\ref{ped}-\ref{slizg}) satisfying conditions \eqref{h} and 
$$\ve \in W^{1,p}(\Omega) \qquad \text{ for all } 1 \leq p < \infty,$$
$${1 \over n_0} \leq \varrho \leq m_0 \qquad \text{almost everywhere in }\Omega$$ 
where ${1 \over n_0} \in \pi^{-1}(- \|G\|_\infty)$, $m_0 \in \pi^{-1}(\|G\|_\infty)$ and $G$ is defined by \eqref{defG}. 
\label{thetw}
\end{tw}
The content of this paper is essentially the proof of Theorem \ref{thetw}. Difficulties involved in the proof include control of some norm of negative part of pressure and construction of appropriate approximation. The second of these issues may be dealt with using the framework proposed by Mucha and Pokorn{\'y} in \cite{mp}, which involves control of $L^\infty$ norm of density imposed by approximate continuity equation. It may be adapted to force boundedness of ${1 \over \varrho}$, i.\,e.\;the condition 
\begin{equation} 
\essinf \varrho >0 .  
\label{essinf} 
\end{equation}
The method, relying on slip boundary conditions (\ref{sciany}, \ref{slizg}), also yields regularity of vorticity
\begin{equation} 
\omega = \rot \ve,  
\label{defomega}
\end{equation} 
and effective viscous flux
\begin{equation} 
G= -(2 \mu + \nu) \dyw \ve + \pi  
\label{defG}
\end{equation} 
in $W^{1,p}(\Omega)$ for any $p<\infty$. This implies that obtained solutions satisfy the following theorem of Lions (provided that $\pi'>0$). 

\begin{tw}[Lions]
Assume that $\fr, \vfr$ are sufficiently regular (e.\,g.\;$C^\infty$). Let $(\varrho, \ve)$ be a weak solution to the system (\ref{ped}-\ref{slizg}) with constitutive equation $\pi = \pi(\varrho)$, where $\pi \in C^1([0, \infty [)$ is non-decreasing, such that $\varrho \in L^\infty(\Omega)$, $\ve, \omega, G \in W^{1,p}(\Omega)$ for all $1 \leq p < \infty$. If condition \eqref{essinf} and  
$$\alpha := \inf \{\pi '(t) t \colon \essinf \varrho \leq t \leq \esssup \varrho \} >0$$
are satisfied, then $\varrho \in C(\Omega)$, $\ve \in C^1(\Omega)$. Furthermore, denoting 
$$\kappa = {\alpha \over 2 \mu + \nu}\|\nabla \ve \|_\infty^{-1},$$
$k = [\kappa]$ (i.\,e.\;the integral part of $\kappa$), $\vartheta = k - \kappa$, there holds
$$\varrho \in C^{k, \vartheta}(\Omega), \quad \ve, \omega, G \in C^{k+1, \vartheta}(\Omega) \quad \text{if } 0 < \vartheta < 1,$$
$$\varrho \in C^{k-1, \eta}(\Omega), \quad \ve, \omega, G \in C^{k, \eta}(\Omega) \quad \text{ for all }\, 0 < \eta < 1 \quad \text{if } \vartheta =1.$$ 
\label{regularnosc}
\end{tw} 

Finally, we mention that the method developed by Mucha and Pokorn{\'y} was also used by them in more complex cases, e.\,g.\;bounded domain in $\mathbb R^3$ with $\gamma \geq 3$ in \cite{mp2}. It should be expected that an analog of Theorem $\thetw$ would also hold in these settings.

The rest of the paper is organised as follows. In the following subsection we give a priori versions of essential estimates of the paper. In section 2, approximate version of the system (\ref{ped}-\ref{cisnienie}), analogous to the one proposed in \cite{mp}, is introduced and the existence of approximate solutions is shown. Next, we show energy estimate independent of parameters of approximation and estimates depending in controlled way on parameters $n_2$ and $m_2$ (forced bounds of approximate $\varrho$). In section 3 we extract convergent sequence $(\varrho_n, \ve_n)$ of approximate solutions and prove that its limit is in fact a weak solution to (\ref{ped}-\ref{slizg}). The key result is that for suitably chosen $(\varrho_n, \ve_n)$
\begin{equation} 
\lim_{n \to \infty} |\{ x \in \Omega \colon \varrho_n(x) \notin [\tfrac{1}{n_0}, m_0]\}| = 0
\label{granicamiary}
\end{equation} 
and $n_2$, $m_2$ may be sufficiently greater than $n_0$ and $m_0$ respectively. 

\subsection{A priori estimates} 
This subsection contains versions of Propositions \ref{apriori} and \ref{odciecie} obtained with assumption on existence of sufficiently regular solutions to (\ref{ped}-\ref{slizg}). Propositions \ref{apriori} and \ref{odciecie} are key parts of the proof of Theorem $\thetw$. Their proofs are ridden with technicalities associated with the form of used approximation. In the a priori regime the ideas behind the estimates are much clearer. However, we need to point out that Proposition \ref{odciecieapriori} is not an a priori estimate in the exact sense, as we ultimately do not prove that the density is necessarily differentiable in any sense, which is assumed in the proof. Therefore to prove the result for actual solutions we will need to act carefully on the level of approximation. 

Here and in the whole paper $C$ will denote any positive constant that may depend on the data, i.\,e.\;$\Omega, \fr, \vfr, h, \mu, \nu, f, \pi$, and may vary on the same page or in the same line. 

\begin{stw}[Energy estimate]
Let $(\varrho, \ve)$ be a regular solution to (\ref{ped}-\ref{slizg}) satisfying \eqref{h}. If $f=0$ and $\Omega$ is symmetric with respect to some $x_0 \in \mathbb R^2$ we assume that $\intO (x-x_0) \times \ve \dd x = 0$. Then
\begin{equation} 
\|\ve\|_{1, 2} + \|\pi_+(\varrho)\|_2 + \|\pi_-(\varrho)\|_2 \leq C.
\end{equation} 
\label{energ}
\end{stw} 
To control the norms $\|\pi_+(\varrho)\|_2$, $\|\pi_-(\varrho)\|_2$ we would like to test the equation \eqref{ped} with vector fields $\boldsymbol \psi_\pm$ such that $\dyw \boldsymbol \psi_\pm = \pi_\pm(\varrho)$. It's convenient to use the Bogovskii operator (see \cite{ns}, Lemma 3.17) to obtain such a field. 

\begin{tw}[Bogovskii operator] 
Let $\Omega \in \mathbb R ^n $ be a bounded domain with Lipschitz boundary, $1 < p < \infty$. There exists a bounded linear operator  
$$\mathcal B \colon \overline{L^p} (\Omega) \to W^{1,p}_0(\Omega)^n,$$ 
where $\overline{L^p} (\Omega) = \{f \in L^p(\Omega) \colon \intO f \dd x = 0 \}$, such that 
$$\dyw \mathcal B (f) = f \text{ a. e. in } \Omega \quad \text{ for all } \, f \in \overline{L^p} (\Omega). $$
\label{bogowski} 
 \end{tw} 

\begin{proof}[Proof of Proposition \ref{energ}]  
Testing momentum equation \eqref{ped} with $\ve$ we obtain 
\begin{equation} 
\intO 2 \mu |\mathbb D (\ve)|^2 + \nu \dyw^2 \ve \dd x + \int_{\partial \Omega} f (\ve \cdot \boldsymbol \tau)^2 \dd \sigma = \intO (\varrho \fr + \vfr) \cdot \ve \dd x. 
\end{equation} 
Using the H\" older and Sobolev inequalities we obtain an estimate 
\begin{equation}
\intO \varrho \fr \cdot \ve \dd x \leq C(\delta)\|\varrho^\delta \|_{2 \gamma \over \delta} \|\varrho^{1 - \delta} \|_{1 \over 1 - \delta} \|\fr\|_\infty \|\ve\|_{1, 2} \leq C(\delta) (1 + \|\pi_+(\varrho)\|_2^{\delta \over \gamma})\|\ve\|_{1, 2} 
\end{equation} 
and therefore, due to the Korn-Poincar\' e inequality
\begin{equation} 
\|\ve\|_{1, 2} \leq C(\delta) (1 + \|\pi_+(\varrho)\|_2^{\delta \over \gamma})
\label{oszv}
\end{equation} 
for any $0 < \delta < 1$. In the case $f = 0$ this only leads to an estimate on $||\ve ||_{W^{1,2}(\Omega)/R(\Omega)}$, where 
\begin{multline} 
R(\Omega) = \{\mathbf u \in W^{1,2}(\Omega) \colon \mathbb D( \mathbf u) = 0 \text{ in }\Omega, \mathbf u \cdot \norm = 0 \text{ on } \brzeg\} \\ = \{\mathbf u \in W^{1,2}(\Omega) \colon \nabla \mathbf u = A \text{ in } \Omega \text{ for some constant }A \in \mathfrak{so}(2), \mathbf u \cdot \norm = 0 \text{ on } \brzeg\} .
\end{multline} 
Clearly, $R(\Omega) \neq \{ 0 \}$ if and only if $\Omega$ is rotationally symmetric with respect to some $x_0 \in \mathbb R^2$. In this case however, $R(\Omega) \cap \{\mathbf u \in W^{1,2}(\Omega) \colon \intO (x-x_0) \times \mathbf u \dd x = 0 \} = \{0 \}$. 

Let $\boldsymbol \psi_+ = \mathcal{B}(\pi_+(\varrho) - {1 \over |\Omega|} \intO \pi_+(\varrho) \dd x)$. Then $\|\boldsymbol \psi_+\|_{1, 2} \leq C \|\pi_+(\varrho)\|_2$. Testing the momentum equation~\eqref{ped} with $\boldsymbol \psi_+$ we get
\begin{multline} 
 \intO \pi_+(\varrho)^2  \dd x + {1 \over |\Omega|}  \intO \pi_+(\varrho)\dd x \intO \pi_-(\varrho) \dd x =  {1 \over|\Omega|}  \left(\intO \pi_+(\varrho)\dd x\right)^2 +   \intO \pi_+(\varrho)\pi_-(\varrho)\dd x  \\- \intO \varrho \ve \otimes \ve \colon \nabla \boldsymbol \psi_+ \dd x +  \mu \intO \nabla \ve : \nabla \boldsymbol \psi_+ \dd x + (\nu + \mu) \intO \dyw \ve \dyw \boldsymbol \psi_+ \dd x  - \intO (\varrho \fr + \vfr) \cdot \boldsymbol \psi_+ \dd x . 
\label{testpi+2} 
\end{multline}
As we control the mass of the fluid, we may estimate the first expression on the r.\,h.\,s.\;of \eqref{testpi+2} using boundedness of $\Omega$,
\begin{multline} 
{1 \over|\Omega|}\left(\intO \pi_+(\varrho)\dd x\right)^2 = {1 \over|\Omega|}\left(\intO \varrho^{1 \over 2} {\pi_+(\varrho) \over \varrho^{1 \over 2}} \dd x\right)^2 \leq h \intO {\pi_+(\varrho)^2 \over \varrho} \dd x \\= h \int_{\{\varrho < 2 h\}} {\pi_+(\varrho)^2 \over \varrho} \dd x + h \int_{\{\varrho \geq 2 h\}} {\pi_+(\varrho)^2 \over \varrho} \dd x \leq C + {1 \over 2} \intO \pi_+(\varrho)^2 \dd x .
\label{oszmasa}
\end{multline}
Next, due to definition of $\pi_\pm$, $\intO \pi_-(\varrho) \pi_+(\varrho) \dd x = 0$. The rest of the terms in \eqref{testpi+2} are estimated using the H\"older and Sobolev inequalities and \eqref{oszv} as follows 
\begin{equation} 
\left|\intO \varrho \ve \otimes \ve : \nabla \boldsymbol \psi_+ \dd x\right| \leq C \|\ve\|_{1 ,2}^2 \|\varrho\|_{2 \gamma} \|\boldsymbol \psi_+ \|_{1,2} \leq C(\delta) \left(1 + \|\pi_+(\varrho)\|_2^{1+ {1 \over \gamma} + 2{\delta \over \gamma}}\right), 
\label{osztestpi+p}
\end{equation} 
\begin{equation} 
\left| \mu \intO \nabla \ve : \nabla \boldsymbol \psi_+ \dd x + (\nu + \mu) \intO \dyw \ve \dyw \boldsymbol \psi_+ \dd x  \right|\leq C \|\ve\|_{1 ,2}  \|\boldsymbol \psi_+ \|_{1,2} \leq C(\delta)  \left(1 + \|\pi_+(\varrho)\|_2^{1 + {\delta \over \gamma}}\right), 
\end{equation} 
\begin{equation}
\left|\intO (\varrho \fr + \vfr) \cdot \boldsymbol \psi_+ \dd x \right| \leq C (1 + \|\varrho\|_{2 \gamma}) \|\boldsymbol \psi_+ \|_{1,2} \leq C \left(1 + \|\pi_+(\varrho)\|_2^{1+ {1 \over \gamma}}\right).
\label{osztestpi+k}
\end{equation} 
By virtue of growth condition \eqref{wzrost2} and convexity of the power function, 
\begin{equation}
 {1 \over |\Omega|} \intO \pi_+(\rho) \geq {a \over |\Omega|} \intO (\varrho^ \gamma\ - \varrho_0^\gamma)\dd x \geq a (h^\gamma - (\tfrac{h}{4})^\gamma)> 0.  
 \label{jensen}
\end{equation} 
Choosing small enough $\delta$ leads to
\begin{equation} 
\int_\Omega \pi_-(\varrho) \dd x + \intO \pi_+(\varrho)^2 \dd x\leq C(\eta)(1+ \|\pi_+(\varrho)\|_2^\eta) 
\end{equation} 
for some $\eta < 2$ and consequently  
\begin{equation} 
\|\pi_-(\varrho)\|_1 + \|\pi_+(\varrho)\|_2 + \|\ve\|_{1, 2} \leq C. 
\label{osz1/2}
\end{equation} 

Further, take $\boldsymbol \psi_- = \mathcal{B}(\pi_-(\varrho) - {1 \over |\Omega|} \intO \pi_-(\varrho) \dd x)$. Then $\|\boldsymbol \psi_-\|_{1, 2} \leq C \|\pi_-(\varrho)\|_2$. Testing the momentum equation~\eqref{ped} with $\boldsymbol \psi_-$ yields 
\begin{multline} 
 \intO \pi_-(\varrho)^2 \dd x + {1 \over |\Omega|}  \intO \pi_+(\varrho)\dd x \intO \pi_-(\varrho) \dd x =  {1 \over|\Omega|}  \left(\intO \pi_-(\varrho) \dd x\right)^2 +   \intO \pi_-(\varrho)\pi_+(\varrho) \dd x \\ - \intO \varrho \ve \otimes \ve \colon \nabla \boldsymbol \psi_- \dd x +  \mu \intO \nabla \ve : \nabla \boldsymbol \psi_- \dd x + (\nu + \mu) \intO \dyw \ve \dyw \boldsymbol \psi_- \dd x - \intO (\varrho \fr + \vfr) \cdot \boldsymbol \psi_- \dd x . 
\label{testpi-} 
\end{multline}
We estimate the first term of the r.\,h.\,s.\;of \eqref{testpi-} using \eqref{osz1/2} and the rest similarly as in (\ref{osztestpi+p}-\ref{osztestpi+k}), obtaining bound on $\|\pi_-(\varrho)\|_2$ which finishes the proof. 
\end{proof} 

\begin{stw} 
Let $(\varrho, \ve)$ be a regular solution to continuity equation \eqref{masa} satisfying impermeability condition \eqref{sciany}. Then 
$${1 \over n_0} \leq \varrho \leq m_0$$ 
where ${1 \over n_0} \in \pi^{-1}(- \|G\|_\infty)$, $m_0 \in \pi^{-1}(\|G\|_\infty)$ and $G$ is defined by \eqref{defG}. 
\label{odciecieapriori}
\end{stw} 
\begin{proof}
Let $k,d >0$. Let $M \in C^1(]0, \infty [)$ satisfy conditions
$$M(t) = \left\{ \begin{array}{rl} 
1 & \text{for }t \leq k,  \\
0 & \text{for }t \geq k + d. \\
\end{array} \right. $$
and $M'(t) < 0$ for $t \in ]k,k+d[$. Let $l \in \mathbb{N}$. 
Testing continuity equation with $M^l({1 \over \varrho})$ leads to 
\begin{multline} 
0 = \intO \varrho \ve \cdot \nabla M^l \left(\tfrac{1}{ \varrho}\right) \dd x= \intO \ve \cdot \nabla \int_\varrho^\infty \tfrac{1}{t} M^{l-1} \left(\tfrac{1}{t}\right) M'\left(\tfrac{1}{t}\right) \dd t\dd x = \\
= \intO \dyw \ve \int_\varrho^\infty t {\dd \over \dd t} M^{l} \left(\tfrac{1}{t}\right) \dd t\dd x 
\end{multline} 
and consequently  
\begin{equation} 
- \intO \pi(\varrho) \int_\varrho^\infty t {\dd \over \dd t} M^{l} \left(\tfrac{1}{t}\right) \dd t\dd x = - \intO G \int_\varrho^\infty t {\dd \over \dd t} M^{l} \left(\tfrac{1}{t}\right) \dd t\dd x. 
\end{equation} 
The derivative ${\dd \over \dd t} M^{l} \left(\tfrac{1}{t}\right)$ might be non-zero only for ${1 \over k+d} \leq t \leq \tfrac{1}{k}$, therefore 
\begin{equation}
  {1 \over k+d} \int_\varrho^\infty {\dd \over \dd t} M^{l} \left(\tfrac{1}{t}\right) \dd t \leq \int_\varrho^\infty t {\dd \over \dd t} M^{l} \left(\tfrac{1}{t}\right) \dd t \leq {1 \over k} \int_\varrho^\infty {\dd \over \dd t} M^{l} \left(\tfrac{1}{t}\right) \dd t. 
\end{equation} 
If $\varrho \geq \tfrac{1}{k}$, the quantity $\int_\varrho^\infty {\dd \over \dd t} M^{l} \left(\tfrac{1}{t}\right) \dd t = 1 - M^{l} \left(\tfrac{1}{\varrho}\right) $ vanishes and consequently, as $\pi(\varrho) \leq 0$ for $\varrho \leq {1 \over k}$, 
\begin{equation} 
- {1 \over k+d} \int_{\{\varrho \leq \tfrac{1}{k}\}} \pi(\varrho) (1 - M^{l} \left(\tfrac{1}{\varrho}\right)) \dd x \leq {1 \over  k} \int_{\{\varrho \leq \tfrac{1}{k}\}} |G| (1 - M^{l} \left(\tfrac{1}{\varrho}\right)) \dd x.
\end{equation} 
Passing to the limit $l \to \infty$ by dominated convergence,  
\begin{equation} 
- {k \over k + d} \pi(\tfrac{1}{k}) |\{\varrho \leq \tfrac{1}{k}\}|\leq - {k \over k + d} \int_{\{\varrho \leq \tfrac{1}{k}\}} \pi(\varrho)  \dd x \leq \|G\|_{\infty}|\{\varrho \leq \tfrac{1}{k}\}|. 
\end{equation} 
Therefore, for $k$ such that $- {k \over k + d} \pi(\tfrac{1}{k}) > \|G\|_\infty$, 
\begin{equation} 
|\{\varrho \leq \tfrac{1}{k}\}| = 0.
\end{equation} 
As $d > 0$ is arbitrary, the lower bound is proven. 

The same method might also be used to obtain upper bound on $\varrho$. Testing continuity equation with $M^l (\varrho)$ yields, after similar calculations, 
\begin{equation}
{k \over k + d} \pi(k) |\{\varrho \geq k\}|\leq \|G\|_{\infty}|\{\varrho \geq k\}|. 
\end{equation} 
Therefore, for $k$ such that $ {k \over k + d} \pi(k) > \|G\|_\infty$, 
\begin{equation} 
|\{\varrho \leq k\}| = 0.
\end{equation}

\end{proof}

\section{Approximation} 
Let $\varepsilon > 0$. We consider the following approximation of (\ref{ped}-\ref{cisnienie}): 
\begin{equation}
\tfrac{1}{2} \dyw (K(\varrho) \varrho \ve \otimes \ve) + \tfrac{1}{2} K(\varrho) \varrho \ve \cdot \nabla \ve - \mu \Delta \ve - (\mu + \nu) \nabla \dyw \ve  + \nabla P(\varrho) = K(\varrho) \varrho \fr + \vfr \qquad \text{in }\Omega,  
\label{aprped} 
\end{equation} 
\begin{equation} 
\dyw (K(\varrho) \varrho \ve) =  \varepsilon \Delta \varrho  - \varepsilon (\varrho - h) \qquad \text{in } \Omega, 
\label{aprmasa} 
\end {equation}
\begin{equation} 
\ve \cdot \mathbf{n} = 0 \qquad \text{on } \partial \Omega, 
\label{aprvnorm}
\end{equation} 
\begin{equation} 
\mathbf{n} \cdot \mathbb{T}(\ve, P(\varrho) ) \cdot \mathbf{\boldsymbol \tau} + f \ve \cdot \boldsymbol \tau = 0 \qquad \text{on } \partial \Omega, 
\label{aprposlizg}
\end{equation} 
\begin{equation} 
{\partial \varrho \over \partial \mathbf{n}} = 0 \qquad \text{on } \partial \Omega, 
\label{aprpnorm}
\end{equation} 
where 
\begin{equation} 
P(\varrho) = \int_0^\varrho \pi_+'(t) K(t)\dd t + \int_\varrho^\infty \pi_-'(t) K(t)\dd t = P_+(\varrho) - P_{-}(\varrho) 
\label{aprp}
\end{equation}
and $K \in C^1(\mathbb{R})$ satisfies conditions 
\begin{equation}
K(t) = \left\{ \begin{array}{rl} 
0 & \text{for }t \leq {1 \over n_2} \\
1 & \text{for } {1 \over n_1} \leq t \leq m_1  \\
0 & \text{for }t \geq m_2 \\
\end{array} \right. 
\end{equation} 
and $K'(t) > 0$ if $t \in ]{1 \over n_2}, {1 \over n_1}[$, $K'(t) < 0$ if $t \in ]m_1,m_2[$, $\tfrac{1}{n_2} < \tfrac{1}{n_1} < \varrho_0 < h < m_1 < m_2$. We fix the (non-important) value of differences $n_2 - n_1 = m_2 - m_1 = h$.

In this section we show the existence of weak solutions to the approximate system (\ref{aprped}-\ref{aprpnorm}) and prove estimates independent (or depending in controlled way) of parameters of approximation allowing extraction of weakly convergent sequence of approximate solutions such that $\varepsilon \to 0^+$. 

\subsection{Existence of solutions to the approximate system}
\begin{tw} 
Let $\Omega \in C^2$, $\varepsilon > 0$, $h >0$. For every $1 \leq p < \infty$ there exists a weak solution $(\varrho, \ve)$ to the approximate system~(\ref{aprped}-\ref{aprpnorm}), $\varrho \in W^{2,p}(\Omega)$, $\ve \in W^{2,p}(\Omega)$ . Furthermore, 
\begin{equation} 
{1 \over n_2} \leq \varrho \leq m_2 \quad \text{in } \Omega, 
\label{aprrhoogr}
\end{equation} 
\begin{equation} 
{1 \over |\Omega|} \intO \varrho \dd x = h .
\label{aprrhosrednia} 
\end{equation} 
\label{aprtw}
\end{tw}

The proof of existence generally follows the ideas in \cite{ns}, Chapter 4.\;and \cite{mp} and is based on application of the Leray-Schauder fixed point theory. However, we do not refrain from using estimates \eqref{aprrhoogr} whenever it is convenient in order to simplify estimates for the fixed point theorem (estimates \eqref{Tfixedosz}, \eqref{oszaprv}). We also notice the issue of continuity of the operator $S$ which seems to be omitted in \cite{ns,mp}. The following simple statement of the Leray-Schauder theory (following e.\,g.\;from 1.4.11.8 in \cite{ns}) is sufficient for our needs. 
\begin{tw}[Leray-Schauder fixed point theorem]
Let $X$ be a Banach space and let $T\colon X \to X$ be a continuous, compact operator on $T$. If the set 
$$\{ u \in X \colon u = t T(u), \quad 0 \leq t \leq 1 \}$$ 
is bounded in $X$, then $T$ has a fixed point in $X$.   
\label{schafer}
\end{tw}

We now define operator $\mathcal T$, whose fixed points are weak solutions to (\ref{aprped}-\ref{aprpnorm}). For $1 \leq p \leq \infty$ let 
\begin{equation}
M^p = \{\we \in W^{1,p}(\Omega) \colon \we \cdot \norm = 0 \text{ on } \partial \Omega\}.  
\end{equation} 
If $f=0$ and $\Omega$ is symmetric with respect to some $x_0 \in \mathbb R^2$, we take instead $M^p = \{\we \in W^{1,p}(\Omega) \colon \we \cdot \norm = 0 \text{ on } \partial \Omega\ \text{ and }\intO (x-x_0) \times \ve \dd x = 0\}$. Let 
\begin{equation}
S \colon M^\infty \ni \ve \mapsto \varrho \in W^{2,p}(\Omega), \qquad 1 \leq p < \infty
\end{equation} 
solve the problem 
\begin{equation} 
- \varepsilon \Delta \varrho = - \varepsilon (\varrho - h) - \dyw (K(\varrho) \varrho \ve) \qquad \text{in } \Omega, 
\label{defS1}
\end{equation} 
\begin{equation} 
{\partial \varrho \over \partial \norm} = 0 \qquad \text{on } \partial \Omega
\label{defS2}
\end{equation} 
and let 
\begin{equation} 
\mathcal{T} \colon M^\infty \ni \ve \mapsto \we \in M^\infty 
\end{equation} 
solve the system
\begin{multline}
- \mu \Delta \we - (\mu + \nu) \nabla \dyw \we = - \tfrac{1}{2} \dyw (K(\varrho) \varrho \ve \otimes \ve) - \tfrac{1}{2} K(\varrho) \varrho \ve \cdot \nabla \ve - \nabla P(\varrho) + K(\varrho) \varrho \fr + \vfr, \\ 
\varrho = S(\ve) \qquad \text{in } \Omega, 
\label{deft1}
\end{multline}
\begin{equation}
\we \cdot \norm = 0 \qquad \text{on } \brzeg, 
\label{deft2}
\end{equation}
\begin{equation} 
\norm \cdot \mathbb{T}(\we, P(\varrho)) \cdot \sty + f\we \cdot \sty = 0 \qquad \text{in } \partial \Omega.  
\label{deft3}
\end{equation} 

\begin{lemat}
Under the assumptions of Theorem \ref{aprtw}, $S$ is a well-defined, continuous operator from $M^\infty$ to $W^{2,p}(\Omega)$ for any fixed $1 \leq p < \infty$. Furthermore, if $\varrho = S(\ve)$, $\|\ve\|_{1, \infty} \leq L$, then
$$\| \varrho \|_{2, p} \leq  C(p, \varepsilon, n_2, m_2, L), \quad \tfrac{1}{n_2} \leq \varrho \leq m_2 ,$$ 
$${1 \over |\Omega|} \intO \varrho \dd x = h . $$
\label{stwrhoapr}
\end{lemat} 
For the proof of Lemma \ref{stwrhoapr} an information about existence and regularity of solutions to the weak Neumann problem  
\begin{equation} 
- \Delta \varrho = \dyw \mathbf b \quad \text{in } \Omega, 
\label{neumann1}
\end{equation}
\begin{equation} 
{\partial \varrho \over \partial \norm} = \mathbf b \cdot \norm \quad \text{on } \brzeg  
\label{neumann2}
\end{equation} 
is required. 
\begin{tw}[Neumann problem]
Let $\Omega$ be a bounded domain of class $C^2$. Let $\mathbf b \in L^p(\Omega)$, $1 < p < \infty$ be a vector field on $\Omega$. There exists a weak solution to (\ref{neumann1}, \ref{neumann2}), i.\,e.\;$\varrho \in W^{1,p}(\Omega)$ such that 
$$\intO \nabla \varrho \cdot \nabla \eta \dd x = - \intO \mathbf b \cdot \nabla \eta \quad \text{for every } \eta \in C^\infty(\mathbb R ^2). $$ 
The weak solution $\varrho$ satisfies an estimate  
$$\|\nabla \varrho \|_p \leq C(p) \| \mathbf b \|_p.$$
Any other solution to (\ref{neumann1}, \ref{neumann2}) has the form $\varrho + c$, $c \in \mathbb R$. 

If moreover $\mathbf b \in W^{k,p} (\Omega)$ and the boundary of $\Omega$ is of class $C^{k+1}$ for $k=1,2,\ldots$, then $\nabla \varrho \in W^{k,p}(\Omega)$ and 
$$\|\nabla \varrho \|_{k, p} \leq C(p) \|\mathbf b\|_{k, p}. $$
\label{neumann}
\end{tw} 
The proof of this fact follows from standard elliptic theory, see \cite{adn}. 

\begin{proof}[Proof of Lemma \ref{stwrhoapr}]
For the purpose of proof we introduce for any $p > 2$ a space $N^p_h = \{\varrho \in W^{1,p}(\Omega) \colon {1 \over |\Omega|} \intO \varrho = h\} $
and operator $T\colon  N^p_h \ni \xi \mapsto \varrho \in N^p_h$ defined by
\begin{equation} 
- \varepsilon \Delta \varrho = - \dyw (K(\xi) \xi \ve) - \varepsilon (\xi - h) \qquad \text{in } \Omega, 
\label{defT1}
\end{equation} 
\begin{equation} 
{\partial \varrho \over \partial \norm} = 0 \qquad \text{on } \partial \Omega   
\label{defT2}
\end{equation} 
in the weak sense for a given $\ve \in M^\infty$. 

Let $\xi \in N^p_h$. Let us define a vector field $\mathbf b \in W^{1,p}(\Omega)$ by $\mathbf b = - \mathcal B (\xi - h) - {1 \over \varepsilon} K(\xi) \xi  \ve$. As $\varrho$ is then a weak solution to (\ref{neumann1}, \ref{neumann2}) with r.\,h.\,s.\;given by $\mathbf b$, we see that $T$ is well defined. Due to Lipschitz continuity of $t \mapsto K(t)t$, $T$ is in fact continuous. Furthermore, as 
\begin{equation}
\nabla (K(\xi) \xi  \ve) = (\xi K'(\xi) \nabla \xi + K(\xi) \nabla \xi) \otimes \ve + K(\xi) \xi \nabla \ve 
\end{equation} 
and $p > 2$ we have an estimate 
\begin{equation} 
\| \mathbf b \|_{1, p} \leq  C (\|\xi\|_p + \tfrac{1}{\varepsilon} (1 + \|\xi \|_{1, p}) \|\xi \|_{1, p} \|\ve\|_{\infty}) .  
\end{equation} 
Due to regularity of Neumann problem, 
\begin{equation} 
\|\nabla T \xi \|_{1, p} \leq C(\varepsilon, L) (1 + \|\xi \|_{1, p})\|\xi\|_{1, p} 
\label{Txi}
\end{equation} 
which implies that $T$ is compact.  

Now, let $\varrho \in N^p_h$ satisfy $\varrho = t T(\varrho)$, i.\,e.\;
\begin{equation} 
- \varepsilon \Delta \varrho = - t \varepsilon (\varrho - h) - t \dyw (K(\varrho) \varrho \ve) \qquad \text{in } \Omega, 
\label{Tt1}
\end{equation} 
\begin{equation} 
{\partial \varrho\ \over \partial \norm} = 0 \qquad \text{on } \partial \Omega .
\label{Tt2}
\end{equation} 
By estimate \eqref{Txi}, $\varrho \in W^{2,p}(\Omega)$. Now we show that the estimate $\tfrac{1}{n_2} \leq \varrho \leq m_2$ holds for all $t$. 

To prove the upper bound on $\varrho$, we test (\ref{Tt1}, \ref{Tt2}) with functions 
\begin{equation} 
((\varrho - m_2)_+ + \delta)^\eta \in W^{1,p}(\Omega), \quad 0<\eta <1, \, 0 < \delta 
\label{regtest}
\end{equation}
obtaining
\begin{multline} 
\varepsilon \intO \nabla\varrho \cdot \nabla((\varrho - m_2)_+ + \delta)^\eta \dd x\\  = - t \varepsilon \intO(\varrho - h)((\varrho - m_2)_+ + \delta)^\eta \dd x + t \intO K(\varrho) \varrho \ve \cdot \nabla((\varrho - m_2)_+ + \delta)^\eta \dd x . 
\end{multline} 
The second term on the r.\,h.\,s.\;equals $0$, while the one on the l.\,h.\,s.\;is non-negative as 
\begin{equation} 
\intO \nabla\varrho \cdot \nabla ((\varrho - m_2)_+ + \delta)^\eta \dd x = \eta \intO |\nabla (\varrho - m_2)_+|^2 ((\varrho - m_2)_+ + \delta)^{\eta - 1} \dd x . 
\end{equation} 
Consequently, 
\begin{equation} 
0 \geq \intO  (\rho - h) ((\varrho - m_2)_+ + \delta)^ \eta \dd x \geq (m_2 - h) \intO ((\varrho - m_2)_+ + \delta)^\eta \dd x . 
\end{equation} 
Passing to the limit $\delta \to 0^+$ and then $\eta \to 0^+$ using dominated convergence yields 
\begin{equation} 
0 \geq (m_2 - h) |\{x \in \Omega \colon \varrho > m_2\}| 
\end{equation} 
which, due to choice of $m_2$, implies
\begin{equation} 
|\{x \in \Omega \colon \varrho > m_2\}| = 0 .
\end{equation} 
The proof that $|\{x \in \Omega \colon \varrho < {1 \over n_2} \}| = 0$ is analogous. 

Now, let
\begin{equation} 
\mathbf b = -  t \mathcal B (\varrho - h) - \tfrac{t}{\varepsilon} K(\varrho) \varrho \ve. 
\end{equation} 
Then $\varrho$ is a weak solution to the Neumann problem (\ref{neumann1}, \ref{neumann2}) and therefore  
\begin{equation} 
\|\nabla \varrho \|_p \leq C \| \mathbf b \|_p \leq C  (\|\varrho\|_p + \tfrac{1}{\varepsilon} \|\ve\|_{1, \infty} \|\varrho\|_p) \leq C(p, \varepsilon, m_2, L) , 
\end{equation} 
and consequently 
\begin{equation} 
\| \varrho\|_{1, p} \leq  C(p, \varepsilon, m_2, L) .
\label{Tfixedosz}
\end{equation} 

By the Leray-Schauder theorem, there exists a fixed point $\varrho \in W^{2,p}(\Omega)$ of the operator $T$. Therefore, $S$ is well defined (into $W^{2,p}(\Omega)$ for any $1 \leq p < \infty$ as $\Omega$ is bounded). The inequalities (\ref{Txi}, \ref{Tfixedosz}) imply estimate on the norm $\|\varrho\|_{2, p}$. Let now $\varrho_k = S(\ve_k)$, $\ve_k \to \ve$ in $W^{1,\infty}(\Omega)$. Then 
\begin{equation} 
- \varepsilon \Delta (\varrho_k - \varrho) = - \varepsilon (\varrho_k - \varrho) - \dyw ((K(\varrho_k) \varrho_k - K(\varrho)\varrho)\ve_k + K(\varrho)\varrho(\ve_k - \ve)) \qquad \text{in } \Omega, 
\label{contS1}
\end{equation} 
\begin{equation} 
{\partial (\varrho_k - \varrho) \over \partial \norm} = 0 \qquad \text{on } \partial \Omega.
\label{contS2}
\end{equation} 
Testing (\ref{contS1}, \ref{contS2}) with $\mathbf 1_{\varrho_k - \varrho > 0}$ and $\mathbf 1_{\varrho_k - \varrho < 0}$ (regularised as in \eqref{regtest}) implies convergence $\varrho_k \to \varrho$ in $L^1(\Omega)$. Now, as $\varrho_k$ is bounded in $W^{2, p}(\Omega)$ we have e.\,g.\;strong convergence in any $L^p(\Omega)$ and in consequence, due to Theorem \ref{neumann}, in $W^{2,p}$, which means continuity of $S$. 
\end{proof}

We now show the existence of a solution to the system (\ref{aprped}-\ref{aprpnorm}) applying the Leray-Schauder theorem to the operator $\mathcal{T}$. To this point we use the information on existence and regularity of solutions to the Lam\'e system. 
\begin{tw}[Lam\'e system] 
Let $1 < p < \infty$, $\Omega \in C^2$, $\mathbf f \in (M^p)^*$, $\mu > 0$, $2 \mu + 3\nu >0$. Then there exists unique $\mathbf w \in M^p$ such that 
$$- \mu \Delta \we - (\mu + \nu) \nabla \dyw \we = \mathbf f \qquad \text{in } \Omega, $$
$$\we \cdot \norm = 0 \qquad \text{on } \brzeg, $$
$$\norm \cdot \mathbb{T}(\we, P(\varrho)) \cdot \sty + f\we \cdot \sty = 0 \qquad \text{on } \partial \Omega $$
and
$$\|\we \|_{1, p} \leq C(p) \| \mathbf f \|_{(M^p)^*}.$$
If moreover $\mathbf f \in L^p(\Omega)$, then $\we \in W^{2,p}(\Omega)$ and 
$$\|\we \|_{2, p} \leq C(p) \|\mathbf f \|_p. $$
\label{lame}
\end{tw}
The proof of this theorem relies on ellipticity of the Lam\'e problem \cite{ns, mp}. 

\begin{proof}[Proof of Theorem \ref{aprtw}] 
Let $\mathbf f$ be the r.\,h.\,s.\;of \eqref{deft1}. For any $\ve \in M^\infty$, by Lemma \ref{stwrhoapr},
\begin{equation}
\|\mathbf f \|_p \leq C(p,\varepsilon, n_2, m_2, \|\ve\|_{1, \infty}) \quad \forall \, 1 \leq p < \infty
\end{equation} 
and due to regularity of the Lam\'e problem (\ref{deft1}-\ref{deft3}), 
\begin{equation} 
\|\mathcal T (\ve) \|_{2, p} \leq C(p,\varepsilon, n_2, m_2, \|\ve\|_{1, \infty}) 
\end{equation} 
which proves (as $S$ is continuous) that $\mathcal T$ is well defined and compact. 

We now estimate the norm of $\ve \in M^\infty$ that satisfy 
\begin{equation} 
\ve = t \mathcal T (\ve), \quad t \in [0,1], 
\label{ls}
\end{equation} 
that is
\begin{multline}
- \mu \Delta \ve - (\mu + \nu) \nabla \dyw \ve = - \tfrac{1}{2} t \dyw (K(\varrho) \varrho \ve \otimes \ve) - \tfrac{1}{2}t K(\varrho) \varrho \ve \cdot \nabla \ve - t \nabla P(\varrho) + t (K(\varrho) \varrho \fr + \vfr) =: \mathbf f_t, \\
\varrho = S(\ve) \quad \text{in } \Omega, 
\label{ls2}
\end{multline}
\begin{equation}
\ve \cdot \norm = 0 \quad \text{on } \brzeg, 
\end{equation}
\begin{equation} 
\norm \cdot \mathbb{T}(\ve, P(\varrho)) \cdot \sty + f\ve \cdot \sty = 0 \quad \text{on } \partial \Omega.  
\end{equation} 
Testing \eqref{ls} with velocity, 
\begin{multline} 
2\mu \intO |\mathbb{D}(\ve)|^2 \dd x + \nu \intO |\dyw \ve|^2 \dd x + \int_{\partial \Omega} f |\ve \cdot \mathbf{\boldsymbol \tau}|^2 \dd S - \tfrac{1}{2} t \intO K(\varrho) \varrho (\ve \cdot \nabla \ve) \cdot \ve  \dd x \\+ \tfrac{1}{2} t \intO K(\varrho) \varrho \ve \otimes \ve \colon \nabla \ve \dd x + t \intO \ve \cdot \nabla P(\varrho) \dd x = t \intO (K(\varrho) \varrho \fr + \vfr) \cdot \ve \dd x.
\end{multline} 
By the approximate continuity equation \eqref{aprmasa} satisfied by $\varrho = S(\ve)$, 
\begin{multline} 
\intO \ve \cdot \nabla P(\varrho) \dd x = \intO \ve \cdot \nabla \varrho K(\varrho)\pi'(\varrho)\dd x = \intO K(\varrho) \varrho \ve \cdot \nabla \int_{\varrho_0}^\varrho {\pi'(\xi) \over \xi} \dd \xi \dd x \\
= \varepsilon \intO (\varrho - h  - \Delta \varrho) \int_{\varrho_0}^\varrho {\pi'(\xi) \over \xi} \dd \xi \dd x = \varepsilon \intO  (\varrho - h) \int_{\varrho_0}^\varrho {\pi'(\xi) \over \xi} \dd \xi \dd x + \varepsilon \intO {\pi'(\varrho) \over \varrho}|\nabla \varrho|^2 \dd x 
\end{multline}
Consequently, as $(\ve \cdot \nabla \ve) \cdot \ve = \ve \otimes \ve \colon \nabla \ve$, 
\begin{multline}
2 \mu \intO |\mathbb{D}(\ve)|^2 \dd x + \nu \intO |\dyw \ve|^2 \dd x + \int_{\partial \Omega} f |\ve \cdot \mathbf{\boldsymbol \tau}|^2 \dd S + t \varepsilon \intO {\pi'(\varrho) \over \varrho}|\nabla \varrho|^2 \dd x\\ 
= - t \varepsilon \intO  (\varrho - h) \int_{\varrho_0}^\varrho {\pi'(\xi) \over \xi} \dd \xi \dd x + t \intO (K(\varrho) \varrho \fr + \vfr) \cdot \ve \dd x
\label{testv}
\end{multline}
and therefore, by comparison of signs of $(\varrho - h)$ and $\int_{\varrho_0}^\varrho {\pi'(\xi) \over \xi}$, 
\begin{multline}
2 \mu \intO |\mathbb{D}(\ve)|^2 \dd x + \nu \intO |\dyw \ve|^2 \dd x + \int_{\partial \Omega} f |\ve \cdot \mathbf{\boldsymbol \tau}|^2 \dd S\\ \leq t C(\varepsilon, n_2, m_2)(1 +  \intO (K(\varrho) \varrho \fr + \vfr) \cdot \ve \dd x) .
\end{multline}
We obtain estimate
\begin{equation}
\|\ve\|_{1,2} \leq C(\varepsilon, n_2, m_2) 
\label{oszaprv}
\end{equation} independent of $t\in [0,1]$. In the case $f=0$ we proceed as in the proof of Proposition \ref{apriori}. 

Using the fact that $\varrho$ is a solution to the Neumann problem (\ref{neumann1}, \ref{neumann2}) with 
\begin{equation} 
\mathbf b = - \mathcal B (\varrho - h) - \tfrac{1}{\varepsilon} K(\varrho) \varrho \ve
\end{equation} 
we obtain from \eqref{oszaprv} and bound \eqref{aprrhoogr}
\begin{equation} 
\|\nabla \varrho\|_p \leq C(p, \varepsilon, n_2, m_2)
\label{costam}
\end{equation} 
and consequently
\begin{equation} 
\|\nabla P_+(\varrho)\|_p + \|\nabla P_-(\varrho)  \|_p \leq C(p, \varepsilon, n_2, m_2)  
\label{oszp+-}
\end{equation} 
for every $1 \leq p < \infty$. From (\ref{oszaprv}, \ref{oszp+-}) we obtain an estimate on $\|\mathbf f_t \|_{3 \over 2}$ independent of $t$. Due to regularity for the Lam\'e problem we have estimate on $\|\ve \|_{2, {3 \over 2}}$ which, by embeddings $W^{2,{3 \over 2}}(\Omega) \subset W^{1,6}(\Omega) \subset L^\infty(\Omega)$, gives bound on the norm of $\mathbf f_t$ in $L^6(\Omega)$. This in turn implies
\begin{equation}
\|\ve \|_{1, \infty} \leq C \|\ve \|_{2, 6} \leq C(\varepsilon, n_2, m_2) 
\end{equation} 
which means that $\mathcal T$ satisfies the assumptions of the Leray-Schauder theorem and finishes the proof of \ref{aprtw}. 
\end{proof}

\subsection{Estimates independent of $\varepsilon$} 
In this section we prove a version of Proposition \ref{energ} for solutions to the approximate system (\ref{aprped}-\ref{aprpnorm}). 
\begin{stw}
There exists a constant $L$ independent of $\varepsilon$, $n_1$, $n_2$, $m_1$ and $m_2$ (provided that $m_1$, $n_1$ are sufficiently large) such that for any weak solution $(\varrho, \ve)$ to (\ref{aprped}-\ref{aprpnorm}),
\begin{equation} 
\|\ve\|_{1, 2} + \|P_+(\varrho)\|_2 + \|P_-(\varrho)\|_2 \leq L. 
\end{equation}
\label{apriori}
\end{stw} 
In the proof we use the following technical remark. 
\begin{uwaga} 
For any $1 \leq \eta \leq 2 \gamma$,
$$\|K(\varrho) \varrho\|_ \eta\leq C\left(1 + \|P_+(\varrho)\|_2^{2 (\eta-1) \over (2 \gamma - 1)\eta}\right). $$ 
\label{Krho}
\end{uwaga} 
\begin{proof} 

Due to the H\"older interpolation inequality, 
\begin{equation} 
\|K(\varrho) \varrho\|_\eta \leq \|K(\varrho) \varrho\|_1^{1 - \vartheta} \|K(\varrho)\varrho \|_{2 \gamma}^\vartheta,  
\end{equation}
where $\vartheta = {2 \gamma (\eta-1) \over (2 \gamma - 1)\eta}$. Due to properties of $K(\cdot)$ and growth condition \eqref{wzrost2}, 
\begin{equation}
\intO (K(\varrho)\varrho)^{2\gamma} \dd x \leq \int_{\{\varrho \leq m_1\}} \varrho^{2\gamma} \dd x + \tfrac{m_2^{2\gamma}}{m_1^{2\gamma}}\int_{\{\varrho > m_1\}} m_1^{2\gamma} \dd x \leq C\left(1 + \intO P_+(\varrho)^2 \dd x\right)
\end{equation} 
and consequently 
\begin{equation} 
\|K(\varrho) \varrho\|_ \eta \leq C \left(1 + \| P_+(\varrho) \|_2^{\vartheta \over \gamma}\right) . 
\end{equation} 
\end{proof} 

We now proceed to the proof of Proposition \ref{apriori}. 
\begin{proof}[Proof of Proposition \ref{apriori}] 
Estimating more precisely the r.\,h.\,s.\;of \eqref{testv} for $t=1$ (notice that the first expression on the r.\,h.\,s.\;may be positive only on the set $\{\varrho_0 \leq \varrho \leq h\}$),  
\begin{equation} 
2 \mu \intO |\mathbb{D}(\ve)|^2 \dd x + \nu \intO |\dyw \ve|^2 \dd x + \int_{\partial \Omega} f |\ve \cdot \mathbf{\boldsymbol \tau}|^2 \dd S  \leq C(\eta)(1 + \|\ve\|_{1 , 2} \| K(\varrho) \varrho \|_\eta)
\end{equation}
for $\eta >1$. Remark \ref{Krho} yields 
\begin{equation} 
2 \mu \intO |\mathbb{D}(\ve)|^2 \dd x + \nu \intO |\dyw \ve|^2 \dd x + \int_{\partial \Omega} f |\ve \cdot \mathbf{\boldsymbol \tau}|^2 \dd S  \leq C(\delta)(1  + \|P_+(\varrho) \|_2^\delta)
\label{testvP}
\end{equation}
for arbitrarily small $\delta > 0$. 

Next we test \eqref{aprped} with $\mathcal B \left(P_\pm(\varrho) - {1 \over |\Omega|} \intO P_\pm(\varrho) \dd x \right)$ in order to control $\|P_\pm(\varrho)\|_2$. With Remark \ref{Krho} and estimate \eqref{testvP} at our disposal we may proceed as in the proof of Proposition \ref{energ} and obtain the assertion. Here we only present in detail the estimate on ${1 \over |\Omega|} \intO P_+(\varrho) \dd x$ from below (needed to control $\intO P_-(\varrho) \dd x$):  
\begin{equation} 
\tfrac{1}{|\Omega|}\intO P_+(\varrho) \dd x \geq \tfrac{a}{|\Omega|} \intO (\min(\varrho, m_1)^\gamma - \varrho_0^\gamma)\dd x \geq a \left(\tfrac{m_1^\gamma}{m_2^\gamma}\tfrac{1}{|\Omega|}\intO \varrho^\gamma \dd x - \varrho_0^\gamma \right) \geq a(\tfrac{1}{2^\gamma} - \tfrac{1}{4^\gamma})h^\gamma > 0 . 
\end{equation} 
\end{proof} 

Using regularity of the Lam\'e system (Theorem \ref{lame}) we may obtain better estimates on the gradient of $\ve$. 
\begin{wniosek} 
Let $(\varrho, \ve)$ be a solution to the system (\ref{aprped}-\ref{aprpnorm}). Then, for any $q>2$, 
$$\| \ve \|_{1 ,q} \leq C(L)\left(1 + \pi_+(m_2)^{1-{2 \over q}} + \pi_-(\tfrac{1}{n_2})^{1 - {2 \over q}}\right) .$$
\label{wniosek}
\end{wniosek} 
\begin{proof} 
The proof relies on application of Lemma \ref{lame} to the system (\ref{aprped}, \ref{aprvnorm}, \ref{aprposlizg}). It implies, due to continuity of embedding $L^{2q \over 2 + q} \subset W^{-1, q}$ for any $q >2$,  
\begin{multline} 
\|\ve \| _{1,q} \leq C\left(\| K(\varrho) \varrho \ve \otimes \ve \|_q + \|K(\varrho) \varrho \ve \cdot \nabla \ve\|_{2q \over 2+q} +\|P_+(\varrho) \| _q + \|P_-(\varrho) \| _q + \|K(\varrho) \varrho \fr + \vfr\|_{2q \over 2+q}\right)
\end{multline}
Estimating respective terms of the above inequality using Proposition \ref{apriori} and Remark \ref{Krho} yields 
\begin{multline*}
\| K(\varrho) \varrho \ve \otimes \ve \|_q + \|K(\varrho) \varrho \ve \cdot \nabla \ve\|_{2q \over 2+q} \leq C \|K(\varrho) \varrho\|_{q + 2 \gamma - 2} \|\ve \|^2_{1, 2} \leq C \|\ve \|^2_{1, 2} \|K(\varrho) \varrho\|_{2 \gamma}^{2 \gamma \over q + 2 \gamma - 2} C m_2^{q - 2 \over q + 2 \gamma - 2}\\ \leq C(L)\left(1 + \pi_+(m_2)^{q-2 \over \gamma(q + 2 \gamma -2)}\right) \leq C(L)\left(1 + \pi_+(m_2)^{1-{2 \over q}}\right), 
\end{multline*}

$$\|P_+(\varrho)\|_q \leq \| P_+(\varrho)\|_2^{2 \over q} \| P_+(\varrho)\|_\infty ^{1 - {2 \over q}} \leq C(L) \pi_+(m_2)^{1 - {2 \over q}}, $$
$$\|P_-(\varrho)\|_q \leq \| P_-(\varrho)\|_2^{2 \over q} \| P_-(\varrho)\|_\infty ^{1 - {2 \over q}} \leq C(L) \pi_-(\tfrac{1}{n_2})^{1 - {2 \over q}}, $$
$$\|K(\varrho) \varrho \fr + \vfr\|_{2q \over 2+q} \leq  C(L),$$ 
which finishes the proof. 
\end{proof} 

We notice following simple statement concerning the speed of blow-up of $\|\nabla \varrho \|_2$. 
\begin{stw} 
Let $(\varrho, \ve)$ be a solution to the system (\ref{aprped}-\ref{aprpnorm}). Then 
\begin{equation} 
\sqrt \varepsilon \| \nabla \varrho \|_2 \leq C(L, m_2). 
\end{equation}
\label{rhoepsilon} 
\end{stw} 
\begin{proof} 
Multiplying the equation \eqref{aprped} by $\varrho$ and integrating over $\Omega$ yields 
\begin{equation} 
\varepsilon \intO |\nabla \varrho |^2 \dd x = \varepsilon h \intO K(\varrho) \varrho \dd x - \varepsilon \intO \varrho^2 \dd x - \intO K(\varrho) \varrho \ve \cdot \nabla \varrho \dd x. 
\end{equation} 
It suffices to estimate the last term on the r.\,h.\,s.\; 
\begin{multline} 
\intO K(\varrho) \varrho \ve \cdot \nabla \varrho \dd x = \intO \ve \cdot \nabla \left( \int_0^\varrho K(t) t \dd t \right) \dd x = - \intO \dyw \ve \left( \int_0^\varrho K(t) t \dd t \right) \dd x \leq C(L, m_2) .
\end{multline} 
\end{proof} 

In order to pass to the limit we need estimates on vorticity $\omega$ defined by 
\begin{equation} 
\omega = \rot \ve = {\partial \ve_2 \over \partial x_1} - {\partial \ve_1 \over \partial x_2}.
\end{equation} 
This is the moment where slip boundary condition \eqref{slizg} is essential, as it guarantees well-posedness of the following problem for vorticity which holds in weak sense
\begin{equation} 
- \mu \Delta \omega = - \tfrac{1}{2} \rot \dyw ((K(\varrho) \varrho) \ve \otimes \ve ) \\ - \tfrac{1}{2} \rot (K(\varrho) \varrho \ve \cdot \nabla \ve) + \rot (K(\varrho) \varrho \fr + \vfr) \quad \text{in } \Omega, 
\label{omegaapr1}
\end{equation} 
\begin{equation} 
\omega = (2 \chi - {f \over \mu} ) \ve \cdot \boldsymbol \tau \quad \text{on } \brzeg, 
\label{omegaapr2}
\end{equation} 
where $\chi$ is the curvature of $\brzeg$ (see \cite{mr} for derivation of boundary condition \eqref{omegaapr2} from (\ref{sciany}, \ref{slizg})). Representing vorticity as sum $\omega = \omega_1 + \omega_2$, where 
\begin{equation} 
-\mu \Delta \omega_1 = -\tfrac{1}{2} \rot \dyw ((K(\varrho) \varrho) \ve \otimes \ve ) \quad \text{in } \Omega, 
\label{omega11}
\end{equation} 
\begin{equation} 
\omega_1 = 0 \quad \text{on } \brzeg, 
\label{omega12}
\end{equation} 
\begin{equation} 
- \mu \Delta \omega_2 =- \tfrac{1}{2} \rot (K(\varrho) \varrho \ve \cdot \nabla \ve)
 + \rot (K(\varrho) \varrho \fr ) \quad \text{in } \Omega,
\label{omega21}
\end{equation}
\begin{equation} 
\omega_2  = (2 \chi - {f \over \mu} ) \ve \cdot \boldsymbol \tau \quad \text{on } \brzeg, 
\label{omega22}
\end{equation} 
we obtain the following assertion. 
\begin{stw} 
Let $\omega_1$, $\omega_2$ as in (\ref{omega11}-\ref{omega22}). Then 
$$\|\omega_1\|_q \leq C(n_2, m_2, q), \quad \|\omega_2 \|_{1 , q} \leq C(n_2, m_2, q) \quad \text{ for all } 1 \leq q < \infty .$$ 
\label{stwomega}
\end{stw} 
\begin{proof} 
The proof follows from standard elliptic theory (see \cite{mp}, \cite{adn}) and a calculation using estimates from Proposition \ref{apriori} and Corollary \ref{wniosek}. 
\end{proof}

\section{Passage to the limit} 
In this chapter we prove the main theorem \ref{thetw} by passing to the limit $\varepsilon \to 0$. We denote by $(\varrho_\varepsilon, \ve_\varepsilon)$ the solution to system (\ref{aprped}-\ref{aprpnorm}) obtained in Theorem \ref{aprtw}. Propositions \ref{apriori}, \ref{wniosek} and \ref{rhoepsilon} from previous chapter give following estimates independent of $\varepsilon$: 
\begin{equation} 
\|P_+(\varrho_\varepsilon) \|_2 + \|P_-(\varrho_\varepsilon)\|_2 + \|\ve_\varepsilon\|_{1, 2} \leq C , \quad \|\varrho_\varepsilon \|_\infty \leq m_2, \quad \| \tfrac{1}{\varrho_\epsilon} \|_\infty \leq n_2, 
\label{oszacowanie1}
\end{equation} 
\begin{equation} 
\sqrt \varepsilon \|\nabla \varrho_\varepsilon\|_2 \leq C(m_2), \quad \|\ve_\varepsilon\|_{1, q} \leq C(n_2, m_2, q) \quad \text{for any } 1 \leq q < \infty .
\label{oszacowanie2} 
\end{equation} 
The Banach-Alaoglu theorem then implies the existence of weak limits of some sequences 
\begin{equation} 
\varrho_{\varepsilon_n}\overset{*} \rightharpoonup \varrho \quad \text{in } L^\infty(\Omega), \qquad \ve_{\varepsilon_n} \rightharpoonup \ve \quad \text{in } W^{1,q}(\Omega). 
\label{przejscierhove} 
\end{equation}
Due to compact embedding $W^{1,q}(\Omega) \subset C(\overline{\Omega})$, $\ve_{\varepsilon_n} \to \ve$ in $C(\overline{\Omega})$ and in any $L^p(\Omega)$, $1 \leq p \leq \infty$. These observations allow us to pass to the limit $\varepsilon \to 0^+$ in the weak formulation of (\ref{aprped}-\ref{aprpnorm}). We denote by bar over nonlinear expression the weak limit (in appropriate space) of a sequence of its approximations. Firstly, we notice  
\begin{uwaga} The equation 
$$\overline{K(\varrho) \varrho} \ve \cdot \nabla \ve = \dyw(\overline{K(\varrho) \varrho} \ve \otimes \ve) = \ve \cdot \overline{K(\varrho)\varrho \nabla \ve}$$ 
holds in weak sense.
\label{uwaga}
\end{uwaga} 
\begin{proof} 
The first equality follows from a calculation involving \eqref{gr2}. Performing similar calculation for approximate solutions yields
\begin{multline} 
\dyw (K(\varrho_{\varepsilon_n}) \varrho_{\varepsilon_n} \ve_{\varepsilon_n} \otimes \ve_{\varepsilon_n}) =  \ve_{\varepsilon_n} \cdot K(\varrho_{\varepsilon_n}) \varrho_{\varepsilon_n} \nabla \ve_{\varepsilon_n} + \dyw(K(\varrho_{\varepsilon_n}) \varrho \ve_{\varepsilon_n}) \ve_{\varepsilon_n} \\
= \ve_{\varepsilon_n} \cdot K(\varrho_{\varepsilon_n}) \varrho_{\varepsilon_n} \nabla \ve_{\varepsilon_n} + \varepsilon_n (h - \varrho_{\varepsilon_n} + \Delta \varrho_{\varepsilon_n}) .
\label{Krhoveve}
\end{multline} 
The assertion follows by passing to the limit in \eqref{Krhoveve}. 
\end{proof} 
By virtue of this remark, $(\varrho, \ve)$ satisfies (in weak sense analogous to Definition \ref{thedef}) the system 
\begin{equation} 
\overline{K(\varrho) \varrho}\ve \cdot \nabla \ve - \mu \Delta \ve - (\mu + \nu) \nabla \dyw \ve 
+ \nabla\overline{P(\varrho)} = \overline{K(\varrho) \varrho} \fr \quad \text{in } \Omega. 
\label{gr1} 
\end{equation}
\begin{equation} 
\dyw(\overline{K(\varrho) \varrho} \ve) = 0 \quad \text{in } \Omega, 
\label{gr2}
\end{equation} 
\begin{equation}
\ve \cdot \norm = 0 \qquad \text{on } \partial \Omega, 
\label{gr3}
\end{equation} 
\begin{equation} 
\norm \cdot \mathbb{T}(\ve, \overline{P(\varrho)}) \cdot \sty + f\ve \cdot \sty = 0 \qquad \text{on } \partial \Omega, 
\label{gr4}
\end{equation} 
where $\overline{P(\varrho)} = \overline{P_+(\varrho)} + \overline{P_-(\varrho)}$. To prove that $(\varrho, \ve)$ is a weak solution to (\ref{ped}-\ref{slizg}) it suffices to show that $\overline{ K(\varrho) \varrho } = K(\varrho) \varrho$ and $\overline{ P(\varrho) } =  \pi(\varrho)$. The estimates from Proposition \ref{odciecieapriori} will be essential in dealing with these limits. To use them, we need to control the $L^\infty$ norm of effective viscous flux $G$. To obtain it we use geometric definition of vector Laplacian of $\ve$ 
\begin{equation} 
\Delta \ve = \nabla \dyw \ve - \nabla^\perp \rot \ve 
\label{veclap}
\end{equation} 
which holds in weak sense for $\ve \in W^{1,2}(\Omega)$. Decomposing $\Delta \ve$ in equation \eqref{gr1} this way yields
\begin{equation} 
\nabla(-(2\mu + \nu) \dyw \ve + \overline{P(\varrho)}) = \mu \nabla^\perp \omega + \overline{K(\varrho) \varrho} \fr - \overline{K(\varrho) \varrho} \ve \cdot \nabla \ve. 
\label{decomp}
\end{equation} 
We denote
\begin{equation} 
\overline{G} : = -(2\mu + \nu) \dyw \ve + \overline{P(\varrho)} .
\label{defGbar}
\end{equation} 

\begin{lemat} 
The following estimates hold
\begin{equation} 
\|\overline G\| _2 \leq C,
\label{g2} 
\end{equation} 
\begin{equation} 
\|\omega\|_{1, q} \leq C(q, n_2, m_2), \quad \|\overline G\|_{1, q} \leq C(q, n_2, m_2) \quad \text{ for any } 1 \leq q < \infty, 
\end{equation}
\begin{equation} 
\|\overline G\| _\infty \leq C(q)\left(1 + \pi_+(m_2)^{1-{2 \over q}} + \pi_-(\tfrac{1}{n_2})^{1 - {2 \over q}}\right)^2 m_2  \quad \text{ for any } 2 < q < \infty. 
\label{ginfty} 
\end{equation} 
\label{ogr} 
\end{lemat}
\begin{proof} 
The estimate \eqref{g2} follows from 
\begin{equation} 
\| \overline G \| _2 \leq C(\| \nabla \ve \| _2 + \| \overline{P(\varrho)} \|_2) \leq C(L). 
\end{equation}
Passing to the limit in (\ref{omegaapr1}, \ref{omegaapr2}) using Proposition \ref{stwomega} and Remark \ref{uwaga}, we get 
\begin{equation} 
 \mu \Delta \omega = \rot ( \overline{K(\varrho) \varrho} \ve \cdot \nabla \ve ) + \rot (  \overline{K(\varrho) \varrho} \fr ) \qquad \text{in } \Omega, 
\label{gromega1}
\end{equation}
\begin{equation} 
\omega = (2 \chi - {f \over \mu} ) \ve \cdot \boldsymbol \tau \qquad \text{on } \partial \Omega 
\label{gromega2}
\end{equation}  
in weak sense. For any $q$ we have $\|\overline{K(\varrho) \varrho} \fr \| _q \leq C m_2$. By virtue of Corollary \ref{wniosek}, 
\begin{equation} 
\|\overline{K(\varrho) \varrho} \ve \cdot \nabla \ve \| _q \leq C(q) \| \nabla \ve \| _q ^2 \| \overline{K(\varrho) \varrho}\|_\infty 
\leq C(q) \left(1 + \pi_+(m_2)^{1-{2 \over q}} + \pi_-(\tfrac{1}{n_2})^{1 - {2 \over q}}\right)^2 m_2 
\end{equation} 
for any $q > 2$. Due to ellipticity of (\ref{gromega1}, \ref{gromega2}), 
\begin{multline} 
\| \omega \| _{1, q} \leq C(q) \| \ve \| _{1, q} + C(q) \| \overline{K(\varrho) \varrho} \ve \cdot \nabla \ve \| _q + \|  \overline{K(\varrho) \varrho} \fr \|_q \\ 
\leq C(q)\left(1 + \pi_+(m_2)^{1-{2 \over q}} + \pi_-(\tfrac{1}{n_2})^{1 - {2 \over q}}\right)^2 m_2.
\end{multline} 
By \eqref{decomp}, for any $q > 2$, 
\begin{multline} 
\|\nabla \overline G\| _q \leq \| \nabla ^\perp \omega \| _q + \| \overline{K(\varrho) \varrho} \fr \|_q   + \|\overline{K(\varrho) \varrho} \ve \cdot \nabla \ve \| _q \\ 
\leq C(q)\left(1 + \pi_+(m_2)^{1-{2 \over q}} + \pi_-(\tfrac{1}{n_2})^{1 - {2 \over q}}\right)^2 m_2. 
\end{multline} 
As we control $\|\overline G\|_2$, estimate \eqref{ginfty} follows from the Morrey inequality. 
\end{proof} 

Next, we decompose $\Delta \ve_\varepsilon$ in \eqref{aprped} in analogous way 
\begin{multline} 
\nabla(-(2\mu + \nu) \dyw \ve_\varepsilon + P(\varrho_\varepsilon)) = \mu \nabla^\perp \omega_\varepsilon + K(\varrho_\varepsilon) \varrho_\varepsilon \fr \\ - \tfrac{1}{2} K(\varrho_\varepsilon) \varrho_\varepsilon \ve_\varepsilon \cdot \nabla \ve_\varepsilon - \tfrac{1}{2} \dyw(K(\varrho_\varepsilon) \varrho_\varepsilon \ve_\varepsilon \otimes \ve_\varepsilon) . 
\end{multline} 
We denote 
\begin{equation} 
G_\varepsilon = -(2\mu + \nu) \dyw \ve_\varepsilon + P(\varrho_\varepsilon) . 
\end{equation} 

\begin{lemat} 
For a suitable sequence $\varepsilon_n \to 0^+$, $G_{\varepsilon_n} \to \overline G \text{ in } L^2(\Omega).$ 
\label{zbieznoscg}
\end{lemat}

\begin{proof} 
From definitions of $\overline G$ and $G_\varepsilon$, 
\begin{multline} 
\nabla (G_\varepsilon - \overline G) = \mu \nabla^\perp (\omega_\varepsilon - \omega) + (K(\varrho_\varepsilon) \varrho_\varepsilon - \overline{K(\varrho) \varrho}) \fr \\ - \tfrac{1}{2} K(\varrho_\varepsilon) \varrho_\varepsilon \ve_\varepsilon \cdot \nabla \ve_\varepsilon - \tfrac{1}{2} \dyw(K(\varrho_\varepsilon) \varrho_\varepsilon \ve_\varepsilon \otimes \ve_\varepsilon) + \overline{K(\varrho) \varrho} \ve \cdot \nabla \ve .
\label{gradGG}
\end{multline} 
We now analyse convergence of expressions on the r.\,h.\,s.\;of \eqref{gradGG}. Obviously
\begin{equation} 
(K(\varrho_{\varepsilon_n}) \varrho_{\varepsilon_n} - \overline{K(\varrho) \varrho}) \fr \slabo 0 \quad \text{in } L^q(\Omega), \, 1 \leq q < \infty. 
\end{equation} 
for the sequence $\varepsilon_n$ from \eqref{przejscierhove} and, consequently, strongly in $W^{-1,2}(\Omega)$, due to compact embedding $L^2(\Omega) \subset W^{-1,2}(\Omega)$.  
Next, 
\begin{multline} 
- \tfrac{1}{2} K(\varrho_{\varepsilon_n}) \varrho_{\varepsilon_n} \ve_{\varepsilon_n} \cdot \nabla \ve_{\varepsilon_n} - \tfrac{1}{2} \dyw(K(\varrho_{\varepsilon_n}) \varrho_{\varepsilon_n} \ve_{\varepsilon_n} \otimes \ve_{\varepsilon_n}) + \overline{K(\varrho) \varrho} \ve \cdot \nabla \ve \\ 
=\overline{K(\varrho) \varrho} \ve \cdot \nabla \ve - K(\varrho_{\varepsilon_n}) \varrho_{\varepsilon_n} \ve_{\varepsilon_n} \cdot \nabla \ve_{\varepsilon_n}  - \tfrac{1}{2} \dyw(K(\varrho_{\varepsilon_n}) \varrho_{\varepsilon_n} \ve_{\varepsilon_n})  \ve_{\varepsilon_n} . 
\label{dowodgg1}
\end{multline}
The difference of first two terms on the r.\,h.\,s.\;of \eqref{dowodgg1} converges to $0$ weakly in $L^q(\Omega)$, $1 \leq q < \infty$, and therefore also strongly in $W^{-1,2}(\Omega)$, by virtue of Remark \ref{uwaga}. The third term might be transformed using \eqref{aprmasa}:
\begin{equation} 
\dyw(K(\varrho_{\varepsilon_n}) \varrho_{\varepsilon_n} \ve_{\varepsilon_n})  \ve_{\varepsilon_n} ={\varepsilon_n} \Delta \varrho_{\varepsilon_n} \ve_{\varepsilon_n} + {\varepsilon_n} h \ve_{\varepsilon_n}  - {\varepsilon_n} \varrho_{\varepsilon_n} \ve_{\varepsilon_n} .
\label{dowodgg2}
\end{equation} 
From estimate \eqref{oszacowanie2} it follows that the first term on the r.\,h.\,s.\;of \eqref{dowodgg2} converges to $0$ strongly in $W^{-1,2}(\Omega)$. The remaining two terms, identical with two subsequent terms in \eqref{gradGG}, converge even strongly in $L^q(\Omega)$. 

It remains to examine the first term on the r.\,h.\,s.\;of \eqref{gradGG}. The difference $\omega_{\varepsilon_n} - \omega$ is a weak solution to
\begin{multline} 
\Delta (\omega_{\varepsilon_n} - \omega) =  \rot ((K(\varrho_{\varepsilon_n})\varrho_{\varepsilon_n} - \overline{K(\varrho) \varrho} )\fr) \\ 
- \tfrac{1}{2} \rot (K(\varrho_{\varepsilon_n}) \varrho_{\varepsilon_n} \ve_{\varepsilon_n} \cdot \nabla \ve_{\varepsilon_n}) - \tfrac{1}{2} \rot \dyw (K(\varrho_{\varepsilon_n}) \varrho_{\varepsilon_n} \ve_{\varepsilon_n} \otimes \ve_{\varepsilon_n}) + \rot(\overline{K(\varrho) \varrho} \ve \cdot \nabla \ve ) \quad \text{in } \Omega,
\end{multline} 
\begin{equation}  
\omega_{\varepsilon_n} - \omega = (2 \chi - {f \over \nu}) (\ve_{\varepsilon_n} - \ve )\cdot \boldsymbol \tau \quad \text{on } \brzeg.  
\end{equation}
Therefore
\begin{multline}
\|\nabla^\perp (\omega_{\varepsilon_n} - \omega) \|_{-1, 2} \leq  \|(K(\varrho_{\varepsilon_n})\varrho_{\varepsilon_n} - \overline{K(\varrho) \varrho} )\fr\|_{-1,2} \\
+ \left \| - \tfrac{1}{2} K(\varrho_{\varepsilon_n}) \varrho_{\varepsilon_n} \ve_{\varepsilon_n} \cdot \nabla \ve_{\varepsilon_n} - \tfrac{1}{2} \dyw(K(\varrho_{\varepsilon_n}) \varrho_{\varepsilon_n} \ve_{\varepsilon_n} \otimes \ve_{\varepsilon_n}) + \overline{K(\varrho) \varrho} \ve \cdot \nabla \ve \right \| _{-1, 2}  + C \|\ve_{\varepsilon_n} - \ve\|_{2, \brzeg} .
\end{multline} 
The convergence of the first and second terms of the r.\,h.\,s.\;was shown above. The last term converges to $0$ due to compactness of the trace operator $\mathrm{tr}\colon W^{1,2}(\Omega) \to L^2(\brzeg)$ (\cite{kof}, Theorem 6.10.5). We proved that
\begin{equation} 
\|G_{\varepsilon_n} - \overline G\|_{L^2(\Omega)/\{\text{constants}\}} = \|\nabla (G_{\varepsilon_n} - \overline G) \|_{-1,2}  \to 0.
\label{zbieznoscbezstalej}
\end{equation} 
But
\begin{equation} 
\intO (G_{\varepsilon_n} - \overline G) \dd x = \intO \dyw (\ve_{\varepsilon_n} - \ve) \dd x + \intO(P(\varrho_{\varepsilon_n}) - \overline{P(\varrho)}) \dd x \to 0,
\end{equation} 
since the first integral on the r.\,h.\,s.\;vanishes due to boundary conditions, which combined with \eqref{zbieznoscbezstalej} completes the proof. 
\end{proof}

We are now ready to show separation of $\varrho$ from $0$ and $\infty$. 
\begin{stw} 
For appropriately chosen $m_2$ and $n_2$, there exists a sequence $\varepsilon_n \to 0^+$ such that 
\begin{equation} 
\lim_{\varepsilon_n \to 0^+} |\{ x \in \Omega \colon \varrho_{\varepsilon_n}(x) \notin [\tfrac{1}{n_0}, m_0]\}| = 0 
\end{equation} 
\label{odciecie} 
for some $n_0 < n_1$, $m_0 < m_1$. In fact we have
\begin{equation} 
\tfrac{1}{n_0} \in \pi^{-1}(-\|\overline G\|_\infty), \quad m_0 \in \pi^{-1}(\|\overline G\|_\infty). 
\label{odciecieosz}
\end{equation}
\end{stw} 
\begin{proof}
Let $M \in C^1(]0, \infty [)$ satisfy conditions 
$$M(t) = \left\{ \begin{array}{rl} 
1 & \text{for }t \leq k,  \\
0 & \text{for }t \geq k + d. \\
\end{array} \right. $$
and $M'(t) < 0$ for $t \in ]k,k+d[$. Let $l \in \mathbb{N}$. 
We test approximate continuity equation \eqref{aprmasa} with $M^l({1 \over \varrho_\varepsilon})$. Since
\begin{equation} 
\intO \Delta \varrho_\varepsilon M^l\left(\tfrac{1}{\varrho_\varepsilon}\right) \dd x = \varepsilon l \intO M^{l - 1} \left(\tfrac{1}{\varrho_\varepsilon}\right) M'\left(\tfrac{1}{\varrho_\varepsilon}\right) \tfrac{1}{\varrho_\varepsilon^2} |\nabla \varrho_\varepsilon |^2 \dd x \leq 0,
\end{equation} 
we obtain
\begin{multline} 
\intO M^l\left(\tfrac{1}{\varrho_\varepsilon}\right) \dyw (K(\varrho_\varepsilon)\varrho_\varepsilon \ve_\varepsilon)\dd x \\ = - \varepsilon \intO \varrho_\varepsilon M^l\left(\tfrac{1}{\varrho_\varepsilon}\right) \dd x + \varepsilon \intO \Delta \varrho_\varepsilon M^l\left(\tfrac{1}{\varrho_\varepsilon}\right) \dd x  
+ \varepsilon h \intO K(\varrho_\varepsilon) M^l \left(\tfrac{1}{\varrho_\varepsilon}\right) \dd x \\ \leq - \varepsilon \intO \varrho_\varepsilon M^l\left(\tfrac{1}{\varrho_\varepsilon}\right) \dd x+ \varepsilon h \intO K(\varrho_\varepsilon) M^l \left(\tfrac{1}{\varrho_\varepsilon}\right) \dd x =: B_\varepsilon
\end{multline}
where $B_\varepsilon \to 0$ as $\varepsilon \to 0$. Therefore, integrating by parts on the l. h. s., 
\begin{multline} 
B_\varepsilon \geq l \intO M^{l - 1} \left(\tfrac{1}{\varrho_\varepsilon}\right) M'\left(\tfrac{1}{\varrho_\varepsilon}\right) {1 \over \varrho_\varepsilon} K(\varrho_\varepsilon) \ve_\varepsilon \cdot \nabla \varrho_\varepsilon \dd x = \\ - l \intO \ve_\varepsilon \cdot \int_{\varrho_\varepsilon}^\infty M^{l - 1} \left(\tfrac{1}{t}\right) M'\left(\tfrac{1}{t}\right) \tfrac{1}{t} K(t) \dd t \dd x .
\end{multline}
Since $ M'({1 \over \varrho_\varepsilon}) = 0 $ for $\varrho_\varepsilon \leq {1 \over k + d}$, this implies
\begin{multline} 
B_\varepsilon \geq  l \intO \dyw \ve_\varepsilon \int_{\varrho_\varepsilon}^\infty M^{l - 1} \left(\tfrac{1}{t}\right) M'\left(\tfrac{1}{t}\right) \tfrac{1}{t} \dd t \dd x  = - \intO \dyw \ve_\varepsilon \int_{\varrho_\varepsilon}^\infty t {\dd \over \dd t} M^l \left(\tfrac{1}{t}\right) \dd t \dd x.
\end{multline} 
By definition of $G_\varepsilon$ 
\begin{multline} 
- \intO P(\varrho_\varepsilon) \int_{\varrho_\varepsilon}^\infty t {\dd \over \dd t} M^l \left(\tfrac{1}{t}\right) \dd t \dd x  \leq - \intO G_\varepsilon \int_{\varrho_\varepsilon}^\infty t {\dd \over \dd t} M^l \left(\tfrac{1}{t}\right) \dd t \dd x + (2 \mu + \nu)  B_\varepsilon .
\end{multline} 
The inner integrals in above inequality might be nonzero only if ${1 \over k+d} < \varrho_\varepsilon < {1 \over k} $. The expression on the l.\,h.\,s.\;is then non-negative and 
\begin{multline} 
- \tfrac{1}{k+ d} \intO P(\varrho_\varepsilon) \int_{\varrho_\varepsilon}^\infty {\dd \over \dd t} M^l \left(\tfrac{1}{t}\right) \dd t \dd x \leq \tfrac{1}{k} \intO |G_\varepsilon| \int_{\varrho_\varepsilon}^\infty {\dd \over \dd t} M^l \left(\tfrac{1}{t}\right) \dd t \dd x + (2 \mu + \nu)  |B_\varepsilon| .
\end{multline} 
Evaluating inner integrals yields
\begin{multline} 
- \tfrac{k}{k+d} \int_{\{ \varrho_\varepsilon < \tfrac{1}{k}\}} P(\varrho_\varepsilon) (1 - M^l\left(\tfrac{1}{\varrho_\varepsilon}\right)) \dd x \leq \int_{\{ \varrho_\varepsilon < \tfrac{1}{k}\}} |G_\varepsilon| (1 - M^l\left(\tfrac{1}{\varrho_\varepsilon}\right)) \dd x + (2 \mu + \nu)  k |B_\varepsilon| .
\end{multline} 
Monotonicity of $P(\cdot)$ implies 
\begin{multline} 
 - \tfrac{k}{k+d} P(\tfrac{1}{k}) |\{ \varrho_\varepsilon < \tfrac{1}{k}\}| \leq \tfrac{k}{k+d} \| P(\varrho_\varepsilon)\|_{L^2(\{ \varrho_\varepsilon < \tfrac{1}{k}\} )} \| M^l (\varrho_\varepsilon)\|_{L^2(\{ \varrho_\varepsilon < \tfrac{1}{k}\} )} \\ 
+ \| \overline G\|_\infty |\{ \varrho_\varepsilon < \tfrac{1}{k}\}| + C \|\overline G - G_\varepsilon \|_2 + (2 \mu + \nu)  k |B_\varepsilon| .
\end{multline} 
Since $M^l (\tfrac{1}{t}) \to 0$ as $l \to \infty$ for $t < \tfrac{1}{k}$, dominated convergence implies that for any $\delta > 0$ there is some $l = l(\delta, \varepsilon)$ such that 
\begin{equation}    
\| M^l (\varrho_\varepsilon)\|_{L^2(\{ \varrho_\varepsilon < \tfrac{1}{k}\} )} \leq \delta,  
\end{equation} 
and therefore 
\begin{equation} 
\left(- \tfrac{k}{k+d} P(\tfrac{1}{k})  - \|\overline G\|_\infty\right) |\{ \varrho_\varepsilon < \tfrac{1}{k}\}| \leq C \delta + C \|\overline G - G_\varepsilon \|_2 + (2 \mu + \nu)  k |B_\varepsilon| .
\end{equation}  
By similar reasoning (see \cite{mp}), testing the approximate continuity equation with~$M^l(\varrho_\varepsilon)$ yields
\begin{equation} 
\left( \tfrac{k}{k+d} P(k)  - \|\overline G\|_\infty\right) |\{ \varrho_\varepsilon > k\}| \leq C \delta + C \|\overline G - G_\varepsilon \|_2 + (2 \mu + \nu)  \tfrac{1}{k+d} |B_\varepsilon'|
\end{equation} 
where $B_\varepsilon' \to 0$ as $\varepsilon \to 0^+$. By virtue of Lemma \ref{ogr}, 
\begin{equation}
\| \overline G \|_\infty \leq C(q)\left(1 + \pi_+(m_2)^{1-{2 \over q}} + \pi_-(\tfrac{1}{n_2})^{1 - {2 \over q}}\right)^2m_2 
\end{equation}
for any $q >2$. Restricting ourselves to $m_2, n_2$ s. t. $\pi_+(m_2) = \pi_-(\tfrac{1}{n_2})$ yields 
\begin{equation} 
\| \overline G \|_\infty \leq C(q)\left(1 + \pi_+(m_2)^{2(1-{2 \over q}) + {1 \over \gamma}}\right) = C(q) \left(1 + \pi_-(m_2)^{2(1-{2 \over q}) + {1 \over \gamma}}\right)
\end{equation} 
Fixing $q$ sufficiently close to $2$ we may have $2(1-{2 \over q}) + {1 \over \gamma} < 1$. Therefore we may choose $n_2$, $m_2$, $n$ and $m$ so that  
\begin{equation} 
- \tfrac{n}{n+d} P(\tfrac{1}{n})  - \|\overline G\|_\infty =: \kappa > 0, \quad \tfrac{m}{m+d} P(m)  - \|\overline G\|_\infty =: \kappa' > 0 
\label{odciecieosz1}
\end{equation} 
and $n < n_1$, $m < m_1$. Consequently
\begin{equation} 
 |\{ \varrho_\varepsilon < \tfrac{1}{n}\}| \leq C(\kappa) (\delta + \|\overline G - G_\varepsilon \|_2 + |B_\varepsilon|), \quad  |\{ \varrho_\varepsilon > m\}| \leq C(\kappa') (\delta + \|\overline G - G_\varepsilon \|_2 + |B_\varepsilon'|) .
\end{equation} 
Due to Lemma \ref{zbieznoscg} and arbitrary choice of $\delta$ 
\begin{equation} 
\lim_{\varepsilon_n \to 0} |\{ x \in \Omega \colon \varrho_{\varepsilon_n}(x) \notin [\tfrac{1}{n}, m]\}| = 0 .
\end{equation} 
Due to arbitrary choice of $d$, $\kappa$, $\kappa'$ in \eqref{odciecieosz1}, we have \eqref{odciecieosz}
\end{proof} 

\begin{wniosek} 
${1 \over n_0} \leq \varrho \leq m_0$ almost everywhere in $\Omega$. 
\label{odcieciegr} 
\end{wniosek} 
\begin{proof} 
Let $A = \{x \in \Omega \colon  \varrho(x) > m_0\}$. Then
\begin{equation} 
\intO \varrho_{\varepsilon_n} \mathbf 1 _{A} \dd x = \int_{\{ \varrho_{\varepsilon_n} \leq  m_0\}} \varrho_{\varepsilon_n} \mathbf 1 _ { A } \dd x +  \int_{\{ \varrho_{\varepsilon_n} >  m_0\}} \varrho_{\varepsilon_n} \mathbf 1 _ { A } \dd x \leq m_0 |A| + m_2 | \{ \varrho_{\varepsilon_n} >  m_0\} |. 
\end{equation} 
Passing to the limit $\varepsilon_n \to 0^+$, 
\begin{equation} 
m_0 |A | \geq \int \varrho \mathbf 1 _ { A } \dd x, 
\end{equation} 
which is impossible unless $|A| = 0$. The proof of the first inequality of the assertion is analogous. 
\end{proof} 
\begin{wniosek} 
$$\overline{K(\varrho)\varrho} = \varrho.$$ 
\end{wniosek} 
\begin{proof} 
For any $\phi \in L^1(\Omega)$,  
\begin{equation}
\intO \varrho_{\varepsilon_n} K(\varrho_{\varepsilon_n}) \phi \dd x = \intO \varrho_{\varepsilon_n} \phi \dd x + \int_{ \{ \varrho_{\varepsilon_n} > m_1 \} } \varrho_{\varepsilon_n} (K(\varrho_{\varepsilon_n}) - 1) \phi \dd x \to \intO \varrho \phi \dd x . 
\end{equation} 
\end{proof} 

It remains to prove that $\overline{P(\varrho)} = \pi(\varrho)$. 
\begin{lemat} 
$$\intO \overline{P(\varrho) \varrho} \dd x \leq \intO \overline G \varrho \dd x .$$
\label{nierG} 
\end{lemat} 

\begin{proof} 
We test the approximate continuity equation \eqref{aprmasa} with $\log m_2 - \log \varrho_{\varepsilon_n}$. Since 
\begin{equation} 
\intO {\varepsilon_n} \Delta \varrho_{\varepsilon_n} (\log m_2 - \log \varrho_{\varepsilon_n}) \dd x = {\varepsilon_n} \intO |\nabla \varrho_{\varepsilon_n} |^2 {1 \over \varrho_{\varepsilon_n}} \dd x \geq 0, 
\end{equation} 
we obtain
\begin{equation} 
\intO (\dyw(K(\varrho_{\varepsilon_n}) \varrho_{\varepsilon_n} \ve_{\varepsilon_n}) + {\varepsilon_n} (\varrho_{\varepsilon_n} - h) ) (\log m_2 - \log \varrho_{\varepsilon_n}) \dd x \geq 0 ,
\end{equation} 
and therefore
\begin{equation} 
\intO K(\varrho_{\varepsilon_n}) \ve_{\varepsilon_n} \cdot \nabla \varrho_{\varepsilon_n} \dd x \geq {\varepsilon_n}\intO  (h  - \varrho_{\varepsilon_n})  (\log m_2 - \log \varrho_{\varepsilon_n}) \dd x .
\end{equation} 
This implies that
\begin{multline} 
- \intO \varrho_{\varepsilon_n} \dyw \ve_{\varepsilon_n} \dd x  \geq \intO (1 - K(\varrho_{\varepsilon_n})) \ve_{\varepsilon_n} \varrho_{\varepsilon_n} \cdot \nabla \varrho_{\varepsilon_n} \dd x + {\varepsilon_n} \intO (h- \varrho_{\varepsilon_n}) \log {m_2 \over \varrho_{\varepsilon_n}}  \dd x \\ \geq  - \intO \left(\int_1^{\varrho_{\varepsilon_n}} t(1 - K(t) ) \dd t \right) \dyw \ve_{\varepsilon_n} \dd x + {\varepsilon_n} \intO (h-\varrho_{\varepsilon_n}) \log {m_2 \over \varrho_{\varepsilon_n}} \dd x =: B_{\varepsilon_n}.   
\end{multline} 
The function $t(1 - K(t))$ is bounded and nonzero only on set $\{\varrho_{\varepsilon_n} < \tfrac{1}{n_1}\} \cup \{ \varrho_{\varepsilon_n} > m_1 \}$, whose measure converges to $0$ as ${\varepsilon_n} \to 0$, therefore $B_{\varepsilon_n} \to 0$ as ${\varepsilon_n} \to 0$. 

By the definition of $G_{\varepsilon_n}$, 
\begin{equation} 
\intO P(\varrho_{\varepsilon_n}) \varrho_{\varepsilon_n} \dd x + (2 \mu + \nu) B_{\varepsilon_n} \leq \intO G_{\varepsilon_n} \varrho_{\varepsilon_n} \dd x.
\end{equation} 
Passing to the limit, 
\begin{equation} 
\intO \overline{P(\varrho) \varrho} \dd x \leq \intO \overline G \varrho \dd x 
\end{equation} 
due to strong convergence of $G_{\varepsilon_n}$ (Lemma \ref{zbieznoscg}). 
\end{proof} 
If $\varrho$ were smooth, testing the continuity equation with $\log \varrho$ would lead to 
\begin{lemat} 
$$\intO \varrho \dyw \ve \dd x = 0. $$
\label{stwnier} 
\end{lemat} 
As we only have $\varrho \in L^\infty(\Omega)$, we need to invoke Friedrichs lemma about commutators, which we state here in a version taken from \cite{ns} (Lemma 3.1). 
\begin{tw}[Friedrichs lemma] 
Let $\sigma \in L^\infty_\text{loc}(\mathbb R ^n)$, $\mathbf{u} \in W^{1,q}_{\text{loc}}(\mathbb R ^n)$ for some $1 < q < \infty$, $n \geq 2$. Let $S_{\delta}\colon \mathcal{D}'(\mathbb R^n) \to C^\infty(\mathbb{R}^n)$ be the operator of convolution with approximate identity. Then  
$$S_\delta (\mathbf u \cdot \nabla \sigma) - \mathbf u \cdot \nabla S_\delta (\sigma) \to 0 \quad \text{in } L^r_\text{loc}(\mathbb R ^n) \quad \text{ for all } 1 \leq r < q,$$
where $\mathbf u \cdot \nabla \sigma := \dyw(\sigma \mathbf u ) - \sigma \dyw \mathbf u$ in $\mathcal D ' (\mathbb R ^n)$. 
\label{friedrichs}
\end{tw}

\begin{proof}[Proof of lemma \ref{stwnier}]
Let $\sigma \in L^\infty(\mathbb R ^2)$ be an extension of $\varrho$ to whole $\mathbb R ^2$ such that $\sigma = 0 \text{ in } \mathbb R^2 \setminus \Omega $ a. e. and let $\mathbf{u} \in W^{1,q}(\mathbb R ^2)$ be any extension of $\ve$ to whole $\mathbb R ^2$. Then, the continuity equation implies that 
\begin{equation} 
\langle \dyw(\sigma \mathbf u ), \varphi \rangle = - \int_{\mathbb R^n} \sigma \mathbf u \cdot \nabla \varphi \dd x = - \intO \varrho \ve \cdot \nabla \varphi \dd x = 0 
\label{uDs}
\end{equation} 
for any $\varphi \in \mathcal D (\mathbb R ^2)$. Therefore $\mathbf u \cdot \nabla \sigma = - \sigma \dyw {\mathbf u}$, in particular $\mathbf u \cdot \nabla \sigma$ is a function that belongs to $L^p(\mathbb R ^2)$, $1 \leq p < \infty$ and 
\begin{equation} 
S_\delta(\mathbf u \cdot \nabla \sigma) \to \mathbf u \cdot \nabla \sigma \quad \text{in } L^p.  
\end{equation} 
Friedrichs lemma implies now that also 
\begin{equation} 
 \mathbf u \cdot \nabla S_\delta(\sigma) \to \mathbf u \cdot \nabla \sigma \quad \text{in } L^p
\end{equation}
(in particular the family $\mathbf u \cdot \nabla S_\delta(\sigma)$ is bounded in $L^p$). 
As $\Omega$ is a regular bounded domain and $\sigma \geq {1 \over n_0}$ in $\Omega$, estimate ${1 \over S_\delta (\sigma)} < M$ holds for any $\delta$ with some $M >0$, which implies in particular that we may test the continuity equation with $\log S_\delta (\sigma)$ obtaining 
\begin{equation} 
0 = \int_\Omega \varrho \ve \cdot \nabla \log S_\delta (\sigma) \dd x = \intO {\sigma \mathbf u \cdot \nabla S_\delta (\sigma) \over S_\delta (\sigma)} \dd x . 
\end{equation}
We have
\begin{equation}
\left| \intO \mathbf u \cdot \nabla S_\delta(\sigma) \dd x - \intO {\sigma \mathbf u \cdot \nabla S_\delta(\sigma) \over S_\delta(\sigma)} \dd x \right|  \leq M \|S_\delta(\sigma) - \sigma \|_2 \| \mathbf u \cdot \nabla S_\delta(\sigma)  \|_2 \to 0   
\label{fried2}
\end{equation} 
as $\delta \to 0^+$, therefore 
\begin{equation} 
 0 = \lim_{\delta \to 0^+} \intO {\sigma \mathbf u \cdot \nabla S_\delta (\sigma) \over S_\delta (\sigma)} \dd x = \lim_{\delta \to 0^+}  \intO \mathbf u \cdot \nabla S_\delta(\sigma) \dd x = - \intO \sigma \dyw {\mathbf u}
\end{equation}

\end{proof} 

Equipped with Lemmata \ref{nierG} and \ref{stwnier}, we may now use the Minty trick to identify $\overline{P(\varrho)}$ as $\pi(\varrho)$. 
\begin{stw}
$\overline{P(\varrho)} = \pi(\varrho)$ almost everywhere in $\Omega$. In fact, if $\pi'$ is strictly positive, $\varrho_{\varepsilon_n} \to \varrho$ strongly in $L^p(\Omega)$ for any $1 \leq p < \infty$. 
\end{stw} 
\begin{proof} 
Let $\alpha >0$, $\phi \in \mathcal{D}(\Omega)$. Due to monotonicity of $P$, 
\begin{equation} 
\intO (P(\varrho_{\varepsilon_n}) - P(\varrho \pm \alpha \phi))(\varrho_{\varepsilon_n} - (\varrho \pm \alpha \phi)) \dd x \geq 0. 
\end{equation} 
Passing to the limit with ${\varepsilon_n}$, 
\begin{equation} 
\intO \overline{P(\varrho) \varrho} - \overline{P(\varrho)} \varrho \mp \alpha \phi \overline{P(\varrho)} \pm \alpha \phi P(\varrho \pm \alpha \phi) \dd x \geq 0.
\end{equation} 
Lemmata \ref{nierG} and \ref{stwnier} imply 
\begin{equation} 
\intO \overline{P(\varrho) \varrho} \dd x \leq \intO \overline G \varrho \dd x = \intO \overline{P(\varrho)} \varrho \dd x 
\end{equation} 
and consequently 
\begin{equation} 
\pm \intO (P(\varrho \pm \alpha \phi) - \overline{P(\varrho)})\phi  \dd x \geq 0 .
\end{equation} 
Passing to the limit $\alpha \to 0^+$, 
\begin{equation} 
\intO (P(\varrho) - \overline{P(\varrho)})\phi  \dd x = 0 .
\end{equation} 
The du Bois-Reymond lemma implies $\overline{P(\varrho)} = P (\varrho)$ a. e. in $\Omega$. Proposition \ref{odcieciegr} implies now that $\overline{P(\varrho)} = \pi(\varrho)$. Now, if $\pi'$ is strictly positive, there exists $C>0$ such that 
\begin{equation} 
\intO (\pi(\varrho_{\varepsilon_n}) - \pi(\varrho))(\varrho_{\varepsilon_n} - \varrho) \dd x \geq C \intO |\varrho_{\varepsilon_n} - \varrho|^2 \dd x. 
\label{mono}
\end{equation} 
As $\pi(\varrho_{\varepsilon_n}) - \pi(\varrho) = \pi(\varrho_{\varepsilon_n}) - P(\varrho_{\varepsilon_n}) + P(\varrho_{\varepsilon_n}) - P(\varrho)$ and $\pi(\varrho_{\varepsilon_n}) - P(\varrho_{\varepsilon_n})$ is non-zero only on $\{\varrho > m_1\}\cup\{\varrho < \tfrac{1}{n_1}\}$, l.\,h.\,s.\;of \eqref{mono} converges to $0$, i.\,e.\;$\varrho_{\varepsilon_n} \to \varrho$ strongly in $L^2(\Omega)$. As $\varrho_{\varepsilon_n}$ is uniformly bounded in $L^\infty(\Omega)$, we have $\varrho_{\varepsilon_n} \to \varrho$ in any $L^p(\Omega)$, $p <\infty$. 
\end{proof} 

This completes the proof of Theorem \ref{thetw}. Moreover, due to Lemma \ref{ogr} and Proposition \ref{odciecie}, if $\pi'$ is strictly positive, the solution $(\rho, \ve)$ satisfies the assumptions of Theorem \ref{regularnosc}. \qed

\section*{Acknowledgements} 
The author wishes to express his gratitude to Piotr Bogus\l aw Mucha for suggesting the problem and stimulating conversations and also to Ewelina Zatorska for critical reading of the paper and her remarks. 

The work has been partly supported by the MN grant No. IdP2011/000661. 

 \bibliography{lasica}{}

\begin{thebibliography}{10}

\bibitem{adn}
S.~Agmon, A.~Douglis, and L.~Nirenberg.
\newblock Estimates near the boundary for solutions of elliptic partial
  differential equations satisfying general boundary conditions. {I}.
\newblock {\em Comm. Pure Appl. Math.}, 12:623--727, 1959.

\bibitem{bd}
D.~Bresch and B.~Desjardins.
\newblock On the existence of global weak solutions to the {N}avier-{S}tokes
  equations for viscous compressible and heat conducting fluids.
\newblock {\em J. Math. Pures Appl. (9)}, 87(1):57--90, 2007.

\bibitem{dz}
R.~Duan and Y.~Zhao.
\newblock A note on the non-formation of vacuum states for compressible
  {N}avier-{S}tokes equations.
\newblock {\em J. Math. Anal. Appl.}, 311(2):744--754, 2005.

\bibitem{fnp}
E.~Feireisl, A.~Novotn{\'y}, and H.~Petzeltov{\'a}.
\newblock On the existence of globally defined weak solutions to the
  {N}avier-{S}tokes equations.
\newblock {\em J. Math. Fluid Mech.}, 3(4):358--392, 2001.

\bibitem{fprs}
E.~Feireisl, H.~Petzeltov{\'a}, E.~Rocca, and G.~Schimperna.
\newblock Analysis of a phase-field model for two-phase compressible fluids.
\newblock {\em Math. Models Methods Appl. Sci.}, 20(7):1129--1160, 2010.

\bibitem{fsw}
J.~Frehse, M.~Steinhauer, and W.~Weigant.
\newblock The {D}irichlet problem for steady viscous compressible flow in three
  dimensions.
\newblock {\em J. Math. Pures Appl. (9)}, 97(2):85--97, 2012.

\bibitem{failure}
D.~Hoff and D.~Serre.
\newblock The failure of continuous dependence on initial data for the
  {N}avier-{S}tokes equations of compressible flow.
\newblock {\em SIAM J. Appl. Math.}, 51(4):887--898, 1991.

\bibitem{hs}
D.~Hoff and J.~Smoller.
\newblock Non-formation of vacuum states for compressible {N}avier-{S}tokes
  equations.
\newblock {\em Comm. Math. Phys.}, 216(2):255--276, 2001.

\bibitem{jz}
S.~Jiang and C.~Zhou.
\newblock Existence of weak solutions to the three-dimensional steady
  compressible {N}avier-{S}tokes equations.
\newblock {\em Ann. Inst. H. Poincar\'e Anal. Non Lin\'eaire}, 28(4):485--498,
  2011.

\bibitem{ks}
A.~V. Kazhikhov and V.~V. Shelukhin.
\newblock Unique global solution with respect to time of initial-boundary value
  problems for one-dimensional equations of a viscous gas.
\newblock {\em Prikl. Mat. Meh.}, 41(2):282--291, 1977.

\bibitem{kof}
A.~Kufner, O.~John, and S.~Fu{\v{c}}{\'{\i}}k.
\newblock {\em Function spaces}.
\newblock Noordhoff International Publishing, Leyden, 1977.

\bibitem{lw}
M.~Lewicka and S.~J. Watson.
\newblock Temporal asymptotics for the {$p$}'th power {N}ewtonian fluid in one
  space dimension.
\newblock {\em Z. Angew. Math. Phys.}, 54(4):633--651, 2003.

\bibitem{lions}
P.-L. Lions.
\newblock {\em Mathematical topics in fluid mechanics. {V}ol. 2. Compressible
  models}, volume~10 of {\em Oxford Lecture Series in Mathematics and its
  Applications}.
\newblock Oxford University Press, Oxford, 1998.

\bibitem{mp}
P.~B. Mucha and M.~Pokorn{\'y}.
\newblock On a new approach to the issue of existence and regularity for the
  steady compressible {N}avier-{S}tokes equations.
\newblock {\em Nonlinearity}, 19(8):1747--1768, 2006.

\bibitem{mp2}
P.~B. Mucha and M.~Pokorn{\'y}.
\newblock 3{D} steady compressible {N}avier-{S}tokes equations.
\newblock {\em Discrete Contin. Dyn. Syst. Ser. S}, 1(1):151--163, 2008.

\bibitem{mr}
P.~B. Mucha and R.~Rautmann.
\newblock Convergence of {R}othe's scheme for the {N}avier-{S}tokes equations
  with slip conditions in 2{D} domains.
\newblock {\em ZAMM Z. Angew. Math. Mech.}, 86(9):691--701, 2006.

\bibitem{np3}
A.~Novotn{\'y} and M.~Pokorn{\'y}.
\newblock Steady compressible {N}avier-{S}tokes-{F}ourier system for monoatomic
  gas and its generalizations.
\newblock {\em J. Differential Equations}, 251(2):270--315, 2011.

\bibitem{np1}
A.~Novotn{\'y} and M.~Pokorn{\'y}.
\newblock Weak and variational solutions to steady equations for compressible
  heat conducting fluids.
\newblock {\em SIAM J. Math. Anal.}, 43(3):1158--1188, 2011.

\bibitem{np2}
A.~Novotn{\'y} and M.~Pokorn{\'y}.
\newblock Weak solutions for steady compressible {N}avier-{S}tokes-{F}ourier
  system in two space dimensions.
\newblock {\em Appl. Math.}, 56(1):137--160, 2011.

\bibitem{ns}
A.~Novotn{\'y} and I.~Stra{\v{s}}kraba.
\newblock {\em Introduction to the mathematical theory of compressible flow},
  volume~27 of {\em Oxford Lecture Series in Mathematics and its Applications}.
\newblock Oxford University Press, Oxford, 2004.

\bibitem{np4}
M.~Pokorn{\'y}.
\newblock On the steady solutions to a model of compressible heat conducting
  fluid in two space dimensions.
\newblock {\em J. Partial Differ. Equ.}, 24(4):334--350, 2011.

\bibitem{ss}
G.~Seregin and V.~{\v{S}}ver{\'a}k.
\newblock Navier-{S}tokes equations with lower bounds on the pressure.
\newblock {\em Arch. Ration. Mech. Anal.}, 163(1):65--86, 2002.

\bibitem{sun3d}
Y.~Sun, C.~Wang, and Z.~Zhang.
\newblock A {B}eale-{K}ato-{M}ajda blow-up criterion for the 3-{D} compressible
  {N}avier-{S}tokes equations.
\newblock {\em J. Math. Pures Appl. (9)}, 95(1):36--47, 2011.

\bibitem{sun2d}
Y.~Sun and Z.~Zhang.
\newblock A blow-up criterion of strong solutions to the 2{D} compressible
  {N}avier-{S}tokes equations.
\newblock {\em Sci. China Math.}, 54(1):105--116, 2011.

\bibitem{vk}
V.~A. Va{\u\i}gant and A.~V. Kazhikhov.
\newblock On the existence of global solutions of two-dimensional
  {N}avier-{S}tokes equations of a compressible viscous fluid.
\newblock {\em Sibirsk. Mat. Zh.}, 36(6):1283--1316, ii, 1995.

\bibitem{wat}
S.~J. Watson.
\newblock Unique global solvability for initial-boundary value problems in
  one-dimensional nonlinear thermoviscoelasticity.
\newblock {\em Arch. Ration. Mech. Anal.}, 153(1):1--37, 2000.

\bibitem{xy}
Z.~Xin and H.~Yuan.
\newblock Vacuum state for spherically symmetric solutions of the compressible
  {N}avier-{S}tokes equations.
\newblock {\em J. Hyperbolic Differ. Equ.}, 3(3):403--442, 2006.

\bibitem{semi}
E.~Zatorska.
\newblock Analysis of semidiscretization of the compressible {N}avier-{S}tokes
  equations.
\newblock {\em J. Math. Anal. Appl.}, 386(2):559--580, 2012.

\bibitem{ika}
E.~Zatorska.
\newblock On the flow of chemically reacting gaseous mixture.
\newblock {\em J. Differential Equations}, 253(12):3471--3500, 2012.

\end{thebibliography}
 \bibliographystyle{plain}
\end{document}